\newcommand{\Rm}{\mathbb{R}}
\newcommand{\Cm}{\mathbb{C}}

\newcommand{\Sm}{\mathbb{S}}
\newcommand{\ba}{\begin{eqnarray*}}
\newcommand{\ea}{\end{eqnarray*}}
\newcommand{\be}{\begin{equation}}
\newcommand{\ee}{\end{equation}}
\newcommand{\bea}{\begin{eqnarray}}
\newcommand{\eea}{\end{eqnarray}}
\newcommand{\va}{\varphi}
\newcommand{\pp}{\partial}
\newcommand{\vv}[1]{\boldsymbol{\mathrm{#1}}}
\newcommand{\hvv}[1]{\boldsymbol{\hat{\mathrm{#1}}}}
\newcommand{\uv}{\boldsymbol{{\hat{\mathrm{s}}}}}
\newcommand{\uvk}{\boldsymbol{\hat{\mathrm{k}}}}
\newcommand{\pint}{\:\mathcal{P}\!\!\int}
\newcommand{\rrf}[1]{\mathop{\mathcal{R}_{{#1}}}}
\newcommand{\irrf}[1]{\mathop{\mathcal{R}_{{#1}}^{-1}}}

\documentclass[10pt]{iopart}
\usepackage{iopams,amsthm,amssymb,mathrsfs}

\usepackage{graphicx}

\newtheorem{thm}{Theorem}[section]

\newtheorem{prop}[thm]{Proposition}
\newtheorem{defn}[thm]{Definition}
\theoremstyle{remark}\newtheorem{rmk}[thm]{Remark}
\newtheorem{exa}[thm]{Example}

\begin{document}

\title[
]{
An $F_N$ method for the radiative transport equation in three dimensions
}

\author{Manabu Machida}

\address{Department of Mathematics, University of Michigan, 
Ann Arbor, MI 48109, USA}
\ead{mmachida@umich.edu}
\begin{abstract}
The $F_N$ method is an accurate and efficient numerical method for the 
one-dimensional radiative transport equation.  In this paper the $F_N$ method 
is extended to three dimensions using rotated reference frames.  To 
demonstrate the method, the exiting flux from structured illumination 
reflected by a medium occupying the half space is calculated.
\end{abstract}

\maketitle

\section{Introduction}

We consider light propagating in a homogeneous random medium occupying 
the half-space $\Rm^3_+$ ($=\{\vv{r}\in\Rm^3;\,\vv{r}=(\vv{\rho},z),\,
\vv{\rho}\in\Rm^2,\,z>0\}$) with the boundary at $z=0$.  The specific 
intensity $I(\vv{r},\uv)$ ($\vv{r}\in\Rm^3_+$, $\uv\in\Sm^2$) of light obeys 
the following radiative transport equation.
\be
\fl
\cases{
\uv\cdot\nabla I(\vv{r},\uv)+I(\vv{r},\uv)
=\varpi\int_{\Sm^2}p(\uv,\uv')I(\vv{r},\uv')\,d\uv'+S(\vv{r},\uv),
&$z>0$,
\\
I(\vv{r},\uv)=f(\vv{\rho},\uv),
&$z=0$, $\mu\in(0,1]$,
\\
I(\vv{r},\uv)\to0,
&$z\to\infty$,
}
\label{rte}
\ee
where $f(\vv{\rho},\uv)$ is the incident beam and $S(\vv{r},\uv)$ is the 
internal source.  Let $\mu$ and 
$\va$ be the cosine of the polar angle and the azimuthal angle 
of $\uv\in\Sm^2$.  Here $\varpi\in(0,1)$ is the 
albedo for single scattering.  Using the absorption and scattering 
parameters $\mu_a$ and $\mu_s$, we have $\varpi=\mu_s/\mu_t$, where 
$\mu_t=\mu_a+\mu_s$ is the total attenuation.  The above form (\ref{rte}) 
implies that $\vv{r}$ is normalized by $\mu_t$.  Furthermore 
$p(\uv,\uv')$ is the scattering phase function which is normalized as
\[
\int_{\Sm^2}p(\uv',\uv)\,d\uv'=1,\quad\uv\in\Sm^2.
\]
The radiative transport equation or the linear Boltzmann equation governs 
transport processes of noninteracting particles such as neutrons in a reactor 
as well as light propagation in random media such as fog, clouds, and 
biological tissue.

In this paper we will present a numerical method of solving (\ref{rte}) 
by extending the $F_N$ method ($F$ stands for facile) to three dimensions.  
The $F_N$ method first developed by Siewert \cite{Siewert78} is a method of 
obtaining the specific intensity in one dimension making use of orthogonality 
relations of singular eigenfunctions \cite{Case60,Case-Zweifel,
Duderstadt-Martin}.  The use of 
rotated reference frames \cite{Markel04,Panasyuk06,Schotland-Markel07} makes 
it possible to extend the $F_N$ method to three dimensions.

In 1960 Case considered the time-independent one-dimensional radiative 
transport 
equation with isotropic scattering and solved the equation with separation 
of variables by finding singular eigenfunctions \cite{Case60}.  The method 
was soon extended to the case of anisotropic scattering without 
\cite{McCormick-Kuscer65,Mika61} and with 
\cite{McCormick-Kuscer66} azimuthal dependence. Such singular-eigenfunction 
approach is sometimes called Caseology. In this method, solutions to 
the one-dimensional radiative transport equation are given by a superposition 
of singular eigenfunctions.  The existence and uniqueness of such solutions 
were proved \cite{Larsen73,Larsen74,Larsen75,LSZ75}.  In the $F_N$ method, 
there is no need of evaluating singular functions although the fact that 
the specific intensity consists of singular eigenfunctions is used.  
In one dimension, the radiative transport equation 
was solved by the $F_N$ method in the slab geometry for isotropic scattering 
\cite{Grandjean-Siewert79,Siewert-Benoist79} and anisotropic scattering 
without \cite{Devaux-Siewert80,Garcia-Siewert85,Siewert78} and with 
\cite{Garcia-Siewert92,Garcia-Siewert98} azimuthal dependence.  
The method was also extended to multigroup \cite{Garcia-Siewert81}.  
After finding the specific intensity on the boundary, we can further calculate 
the specific intensity inside the medium \cite{Garcia-Siewert85}.  
The uniqueness of the solution to the key $F_N$ equation was proved 
\cite{Larsen82}.  
For isotropic scattering, the three dimensional radiative transport equation 
was solved with the $F_N$ method \cite{Dunn-Siewert85,Siewert-Dunn83} 
using the pseudo-problem \cite{Williams82}, which is based on plane-wave 
decomposition. See the review article by Garcia \cite{Garcia85}.

In 1964 Dede used rotated reference frames to solve the three-dimensional 
radiative transport equation with the $P_N$ method \cite{Dede64}.  Dede 
pointed out that equations in three dimensions reduce to one-dimensional 
equations if reference frames are rotated in the direction of the Fourier 
vector.  Kobayashi developed Dede's calculation and computed coefficients 
in the $P_N$ expansion by solving a three-term recurrence relation 
recursively starting with the initial term \cite{Kobayashi77}.  In 2004 
Markel obtained the coefficients in terms of eigenvalues and eigenvectors of 
the tridiagonal matrix originating from the three-term recurrence relation, 
and showed that the specific intensity can be efficiently computed 
\cite{Markel04}.  With the use of eigenvalues, the relation to Case's method 
became visible.  This new formulation can be viewed as separation 
of variables in which the eigenvalues are separation constants 
\cite{Schotland-Markel07}.  Moreover it was found that any complex unit 
vector can be used to rotate reference frames \cite{Panasyuk06}.  This 
generalization makes it possible to solve boundary value problems in the form 
of plane-wave decomposition \cite{Machida10}.  It was then found that the 
structure of separation of variables implies Case's method in rotated 
reference frames \cite{Machida14}.  Thus the singular-eigenfunction approach 
was extended to three dimensions.  Indeed the method of rotated 
reference frames is a three-dimensional extension of the 
spherical-harmonic expansion 
\cite{Barichello-Garcia-Siewert98,Sanchez-McCormick82} in Caseology. 

The usefulness of the method of rotated reference frames has been numerically 
justified for a two-dimensional rectangular domain \cite{Kobayashi77}, 
a three-dimensional infinite medium \cite{Markel04,Panasyuk06}, 
the slab geometry in three dimensions \cite{Machida10}, 
in flatland \cite{LK11,LK12a,LK13c}, 
in the half-space geometry \cite{LK12c,LK12e,LK13a,LK13b,LK14}, and 
the time-dependent equation in an infinite medium \cite{LK12b,LK12d}.  
The method was also used to experimentally determine optical properties of 
turbid media \cite{Xu-Patterson06a,Xu-Patterson06b}. It is expected that 
more accurate numerical values are obtained if higher terms in the series 
are taken into account. 
Although the method of rotated reference frames is an efficient method, 
the obtained values become unstable when high-degree spherical 
harmonics are used.  The three-dimensional $F_N$ method 
developed in the present paper does not suffer from this instability.

By assuming that scatterers are spherically symmetric, we model $p(\uv,\uv')$ 
as
\be
p(\uv,\uv')
=\frac{1}{4\pi}\sum_{l=0}^L\beta_lP_l(\uv\cdot\uv')
=\sum_{l=0}^L\sum_{m=-l}^l\frac{\beta_l}{2l+1}Y_{lm}(\uv)Y_{lm}^*(\uv'),
\label{phasefunc}
\ee
where $L\ge1$, and $\beta_0=1$, $0<\beta_l<2l+1$ for $l\ge1$.  
Moreover $P_l$ are Legendre polynomials and $Y_{lm}$ are spherical harmonics.  
We introduce the scattering asymmetry parameter $\mathrm{g}$ as 
$\beta_l=(2l+1)\mathrm{g}^l$ ($0<\mathrm{g}<1$).  The Henyey-Greenstein model 
\cite{Henyey-Greenstein41} is obtained in the limit $L\to\infty$.  

Let us define
\[
\tilde{I}(\vv{q},z,\uv)
=\int_{\Rm^2}e^{i\vv{q}\cdot\vv{\rho}}I(\vv{r},\uv)\,d\vv{\rho},\quad
\vv{q}\in\Rm^2.
\]
We similarly define $\tilde{f}(\vv{q},\uv)$ and $\tilde{S}(\vv{q},z,\uv)$.  
Let us express the upper and lower hemispheres as 
$\Sm^2_{\pm}=\{\uv\in\Sm^2;\,\pm\mu>0\}$. 
We expand the Fourier transform of the reflected light  
$\tilde{I}(\vv{q},0,-\uv)$ ($\uv\in\Sm^2_+$) as
\be
\tilde{I}(\vv{q},0,-\uv)
\approx
\sum_{m=-l_{\rm max}}^{l_{\rm max}}
\sum_{\alpha=0}^{\left\lfloor(l_{\rm max}-|m|)/2\right\rfloor}
c_{|m|+2\alpha,m}(\vv{q})Y_{|m|+2\alpha,m}(\uv),
\label{FNexpansion}
\ee
where $l_{\rm max}$ is the highest degree of the expansion 
($l_{\rm max}\ge L$). Only same-parity degrees 
are taken because the three-term recurrence relation of associated Legendre 
polynomials implies that $Y_{lm}$ of opposite-parity $l$ are not independent 
\cite{Garcia-Siewert98,Panasyuk06}. This expansion in (\ref{FNexpansion}) can 
be compared to the $P_N$ method \cite{Case-Zweifel}, but the $F_N$ method is 
more efficient because the spatial dependence of the specific intensity is 
analytically given and the orthogonality relation among three-dimensional 
singular eigenfunctions can be used (see \S\ref{pre:case3d}). 
On the other hand, $\tilde{I}(\vv{q},0,-\uv)$ is given as a linear combination 
of eigenmodes $\rrf{\uvk(\nu,\vv{q})}\Phi_{\nu}^{m'}(\uv)$ \cite{Machida14}, 
for which notations are introduced in \S\ref{pre:case3d}. They satisfy 
orthogonality relations. For simplicity let us assume $S(\vv{r},\uv)=0$. 
Making use of the fact that $\tilde{I}$ contains only decaying modes, we have 
(See (\ref{projection1}) for the general case)
\be
\int_{\Sm^2}\mu\left(\rrf{\uvk(-\xi,\vv{q})}\Phi_{-\xi}^{m'*}(\uv)
\right)\tilde{I}(\vv{q},0,\uv)\,d\uv=0,\qquad\xi>0.
\label{projection0}
\ee
The above equation results in a linear system for $c_{|m|+2\alpha,m}(\vv{q})$. 
The specific intensity of the reflected light is then calculated as
\[
I(\vv{\rho},0,-\uv)
\approx\frac{1}{(2\pi)^2}\int_{\Rm^2}e^{-i\vv{q}\cdot\vv{\rho}}
\sum_{m=-l_{\rm max}}^{l_{\rm max}}\sum_{l=|m|,|m|+2,\dots}
c_{lm}(\vv{q})Y_{lm}(\uv)\,d\vv{q},
\]
where $\mu\in(0,1]$.

\begin{rmk}
Isotropic scattering $\mathrm{g}=0$ is possible.  However we need to change 
the collocation scheme for obtaining $c_{lm}$.  For the sake of simplicity, 
we assume $\mathrm{g}>0$ in this paper.
\end{rmk}

\begin{rmk}
The expansion in (\ref{FNexpansion}) can be compared to the method of 
rotated reference frames, which expands every eigenmode with 
spherical harmonics:
\be
\rrf{\uvk(\nu,\vv{q})}\Phi_{\nu}^{m'}(\uv)
\approx\sum_{l=0}^{l_{\rm max}}\sum_{m=-l}^lc_{lm}^{m'}(\nu)
\rrf{\uvk(\nu,\vv{q})}Y_{lm}(\uv),
\label{mrrfexpansion}
\ee
with some coefficients $c_{lm}^{m'}(\nu)$. This causes 
numerical instability regardless of $f(\vv{\rho},\uv)$ and $S(\vv{r},\uv)$ 
when $l_{\rm max}$ is increased to achieve higher precision. 
For example, let us consider a simple case of $L=0$, $m'=0$, 
$\cos(\va-\va_{\vv{q}})=0$, and $\nu\neq\mu\hat{k}_z(\nu q)$. Noting that 
$\rrf{\uvk(\nu,\vv{q})}Y_{lm}(\uv)=\frac{2l+1}{4\pi}
P_l^m\left(\mu\hat{k}_z(\nu q)\right)\rrf{\uvk(\nu,\vv{q})}e^{im\va}$, we see 
that the right-hand side of (\ref{mrrfexpansion}) is a polynomial of 
$\mu\hat{k}_z(\nu q)$. On the left-hand side, we have 
$\rrf{\uvk(\nu,\vv{q})}\Phi_{\nu}^{m'}(\uv)=\frac{\varpi\nu}{2}
\left[\nu-\mu\hat{k}_z(\nu q)\right]^{-1}=
\frac{\varpi}{2}\left[1+(1/\nu)\mu\hat{k}_z(\nu q)
+(1/\nu)^2\left(\mu\hat{k}_z(\nu q)\right)^2+\cdots\right]$. This series is 
divergent if $\nu-\mu\hat{k}_z(\nu q)<0$. In general, the instability takes 
place due to the same mechanism. Figure \ref{fig0} shows 
the exiting current on the boundary ($z=0$) as a function of $l_{\rm max}$. 
See \S\ref{numerics} for the details.

\begin{figure}[ht]
\begin{center}
\includegraphics[width=0.5\textwidth]{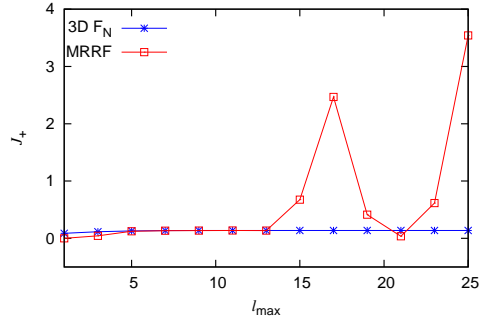}
\end{center}
\caption{
The exitance (\ref{hemisphericJabs}) is plotted as a function of 
$l_{\rm max}$ for $\mu_a=0.05$, $\mu_s=100$, and $\mathrm{g}=0.01$. 
We set $L=l_{\rm max}$.
}
\label{fig0}
\end{figure}

\end{rmk}

The remainder of the paper is organized as follows.  In \S\ref{pre} we 
introduce singular eigenfunctions and rotated reference frames.  In \S\ref{fn} 
we consider the $F_N$ method in three dimensions.  The key $F_N$ equation is 
obtained in (\ref{keyFN}), from which the coefficients $c_{lm}$ in 
(\ref{FNexpansion}) are computed.  
In \S\ref{numerics} the three-dimensional $F_N$ method is 
numerically tested for structured illumination.  Section \ref{concl} is 
devoted to concluding remarks.  Finally structured illumination by 
the method of rotated reference frames is summarized in 
\ref{mrrf}.

\section{Preliminaries}
\label{pre}

To develop the $F_N$ method in three dimensions in \S\ref{fn}, 
we give brief reviews and define our notations in this section. 
In \S\ref{pre:poly}, we introduce polynomials $g_l^m$ and $p_l^m$. 
In \S\ref{case1d}, Case's singular-eigenfunction approach is explained. 
In \S\ref{pre:case3d}, we give a review on singular eigenfunctions in 
three dimensions. 
In \S\ref{pre:mrrf}, it is sketched how the method of rotated reference frames 
is obtained using three-dimensional singular eigenfunctions.

\subsection{Polynomials}
\label{pre:poly}

\begin{defn}
We introduce $h_l$ ($l=0,1,\dots$) as
\[
h_l=2l+1-\varpi\beta_l\chi_{[0,L]}(l),
\]
with $\chi_{[0,L]}(l)$ the step function ($\chi=1$ for $0\le l\le L$ and 
$\chi=0$ otherwise).  
\end{defn}

\begin{defn}[Refs.~\cite{Garcia-Siewert89,Garcia-Siewert90}]
The normalized Chandrasekhar polynomials $g_l^m(\xi)$ ($m\ge0$, $l\ge m$, 
$\nu\in\Rm$) are given by the three-term recurrence relation 
\be
\nu h_lg_l^m(\nu)
=\sqrt{(l+1)^2-m^2}g_{l+1}^m(\nu)+\sqrt{l^2-m^2}g_{l-1}^m(\nu),
\label{chandra1}
\ee
with the initial term
\be
g_m^m(\nu)=\frac{(2m-1)!!}{\sqrt{(2m)!}}=\frac{\sqrt{(2m)!}}{2^{m}m!}.
\label{chandra1init}
\ee
\end{defn}

We note that
\[
g_l^{-m}(\nu)=(-1)^mg_l^m(\nu),\qquad g_l^m(-\nu)=(-1)^{l+m}g_l^m(\nu).
\]
The polynomials $g_l^m$ are obtained if we multiply Chandrasekhar polynomials 
\cite{Chandrasekhar} by $\sqrt{(l-m)!/(l+m)!}$ \cite{Siewert-McCormick97}.

\begin{defn}
The polynomials $p_l^m(\mu)$ ($m\ge0$, $l\ge m$) are introduced as
\be
\fl
p_l^m(\mu)
=(-1)^m\sqrt{\frac{(l-m)!}{(l+m)!}}P_l^m(\mu)(1-\mu^2)^{-m/2}
=\sqrt{\frac{(l-m)!}{(l+m)!}}\frac{d^m}{d\mu^m}P_l(\mu),
\label{deflowercasep}
\ee
where $P_l(\mu)$ is the Legendre polynomial of degree $l$ and $P_l^m(\mu)$ is 
the associated Legendre polynomial of degree $l$ and order $m$.
\end{defn}

We have
\[
p_l^{-m}(\mu)=(-1)^mp_l^m(\mu).
\]
The polynomials satisfy the three-term recurrence relation
\be
\fl
\sqrt{l^2-m^2}p_{l-1}^m(\mu)-(2l+1)\mu p_l^m(\mu)
+\sqrt{(l+1)^2-m^2}p_{l+1}^m(\mu)=0,
\label{precurr}
\ee
with
\[
p_{|m|}^{|m|}(\mu)=\frac{(2|m|-1)!!}{\sqrt{(2|m|)!}}
=\frac{\sqrt{(2|m|)!}}{2^{|m|}|m|!},
\]
and the orthogonality relation
\[
\int_{-1}^1p_l^m(\mu)p_{l'}^m(\mu)\left(1-\mu^2\right)^{|m|}\,d\mu
=\frac{2}{2l+1}\delta_{ll'}.
\]

\subsection{Singular eigenfunctions for one dimension}
\label{case1d}

We will first investigate the one-dimensional homogeneous radiative 
transport equation (\ref{rte1d}) and then consider the three dimensional 
equation (\ref{rte3d}).  Let us begin with
\be
\mu\frac{\pp}{\pp z}I(z,\uv)+I(z,\uv)
=\varpi\int_{\Sm^2}p(\uv,\uv')I(z,\uv')\,d\uv',
\label{rte1d}
\ee
where $z\in\Rm$, $\uv\in\Sm^2$ and $p(\uv,\uv')$ is given 
in (\ref{phasefunc}).  Separated solutions to (\ref{rte1d}) are given by 
\cite{Case60,McCormick-Kuscer66,Mika61}
\be
I(z,\uv)=\Phi_{\nu}^m(\uv)e^{-z/\nu},
\label{separated1d}
\ee
where $\nu\in\Rm$ is a separation constant, $m$ ($|m|\le L$) 
is an integer, and
\be
\Phi_{\nu}^m(\uv)=\phi^m(\nu,\mu)\left(1-\mu^2\right)^{|m|/2}e^{im\va}.
\label{Phidef}
\ee
Here $\phi^m(\nu,\mu)$ satisfies
\[
\int_{-1}^1\phi^m(\nu,\mu)\left(1-\mu^2\right)^{|m|}\,d\mu=1.
\]
By plugging (\ref{separated1d}) into (\ref{rte1d}) we obtain
\be
\left(1-\frac{\mu}{\nu}\right)\Phi_{\nu}^m(\uv)
=\varpi\int_{\Sm^2}p(\uv,\uv')\Phi_{\nu}^m(\uv')\,d\uv'.
\label{Phieq1}
\ee
We multiply (\ref{Phieq1}) by $Y_{l'm'}^*(\uv)$ and integrate both sides 
over $\Sm^2$. By noticing the expression of $p(\uv,\uv')$ in (\ref{phasefunc}) 
and rearranging terms, we obtain
\be
\fl
\nu\left(1-\frac{\varpi\beta_{l'}}{2l'+1}\chi_{[0,L]}(l')\right)
\int_{\Sm^2}Y_{l'm'}^*(\uv)\Phi_{\nu}^m(\uv)\,d\uv
=\int_{\Sm^2}\mu Y_{l'm'}^*(\uv)\Phi_{\nu}^m(\uv)\,d\uv.
\label{chandragint1}
\ee
Using the recurrence relation 
$(2l+1)\mu P_l^m(\mu)=(l+1-m)P_{l+1}^m(\mu)+(l+m)P_{l-1}^m(\mu)$, we see that 
(\ref{chandragint1}) becomes the three-term recurrence relation 
(\ref{chandra1}) for $m'=m$.  That is, we obtain
\be
g_l^m(\nu)
=(-1)^m\sqrt{\frac{(l-m)!}{(l+m)!}}
\int_{-1}^1\phi^m(\nu,\mu)(1-\mu^2)^{|m|/2}P_l^m(\mu)\,d\mu.
\label{chandragint2}
\ee
Noting that $P_m^m(\mu)=(-1)^m(2m-1)!!(1-\mu^2)^{m/2}$ ($m\ge0$), we see that 
(\ref{chandragint2}) satisfies (\ref{chandra1init}).  

Let us rewrite (\ref{Phieq1}) as
\[
\left(1-\frac{\mu}{\nu}\right)\phi^m(\nu,\mu)
=\frac{\varpi}{2}\sum_{l'=|m|}^L\beta_{l'}p_{l'}^m(\mu)g_{l'}^m(\nu).
\]
We define $g^m$ as
\be
g^m(\nu,\mu)=\sum_{l=|m|}^L\beta_lp_l^m(\mu)g_l^m(\nu).
\label{gfuncbpg}
\ee
Singular eigenfunctions $\phi^m(\nu,\mu)$ are thus obtained as
\[
\phi^m(\nu,\mu)=\frac{\varpi\nu}{2}\mathcal{P}\frac{g^m(\nu,\mu)}{\nu-\mu}
+\lambda^m(\nu)\left(1-\mu^2\right)^{-|m|}\delta(\nu-\mu),
\]
where $\mathcal{P}$ denotes the Cauchy principal value.  Here the separation 
constant $\nu$ has discrete values $\pm\nu_j^m$ 
($\nu_j^m>1$, $j=0,1,\dots,M^m-1$) and the continuous spectrum between 
$-1$ and $1$.  The number $M^m$ of discrete eigenvalues depends on $\varpi$ 
and $\beta_l$.  The function $\lambda^m(\nu)$ is given by
\[
\lambda^m(\nu)=1-\frac{\varpi\nu}{2}\pint_{-1}^1\frac{g^m(\nu,\mu)}{\nu-\mu}
(1-\mu^2)^{|m|}\,d\mu.
\]
Discrete eigenvalues satisfy
\[
\Lambda^m(\nu_j^m)=0,
\]
where for $w\in\Cm$
\[
\Lambda^m(w)=1-\frac{\varpi w}{2}\int_{-1}^1\frac{g^m(w,\mu)}{w-\mu}
(1-\mu^2)^{|m|}\,d\mu.
\]
By using $P_l^{-m}=P_l^m(-1)^m(l-m)!/(l+m)!$, we can readily check that 
$g_l^m(\nu)$ in (\ref{chandragint2}) satisfy 
$g_l^{-m}(\nu)=(-1)^mg_l^m(\nu)$.  This implies $\phi^{-m}=\phi^m$.  Singular 
eigenfunctions $\phi^m(\nu,\mu)$ satisfy 
\cite{Case60,McCormick-Kuscer66,Mika61}
\[
\int_{-1}^1\mu\phi^m(\nu,\mu)\phi^m(\nu',\mu)\,d\mu
=\mathcal{N}^m(\nu)\delta_{\nu\nu'},
\]
where the Kronecker delta $\delta_{\nu\nu'}$ is replaced by the Dirac delta 
$\delta(\nu-\nu')$ if $\nu,\nu'$ are in the continuous spectrum. The 
normalization factor $\mathcal{N}^m(\nu)$ is given by
\be
\mathcal{N}^m(\nu)=\cases{
\frac{1}{2}(\nu_j^m)^2g(\nu_j^m,\nu_j^m)\left.\frac{d\Lambda^m(w)}{dw}
\right|_{w=\nu_j^m},&$\nu=\nu_j^m$,
\\
\nu\Lambda^{m+}(\nu)\Lambda^{m-}(\nu)(1-\nu^2)^{-|m|},&$\nu\in(-1,1)$,
}
\label{normalizationN}
\ee
where $\Lambda^{m\pm}(\nu)=\lim_{\epsilon\to0^+}\Lambda^m(\nu\pm i\epsilon)$.

We can numerically obtain the discrete eigenvalues $\nu_j^m$ 
as eigenvalues of a tridiagonal matrix $B(m)$ below.  For $l_B$ 
($\ge L$) and  $m$ ($-L\le m\le L$), the matrix $B(m)$ is given by
\be
B(m)=\left(\begin{array}{ccccc}
0&b_{|m|+1}&0&&\\
b_{|m|+1}&0&b_{|m|+2}&&\\
0&b_{|m|+2}&0&\ddots&\\
&&\ddots&\ddots&b_{l_B}\\
&&&b_{l_B}&0
\end{array}\right),
\label{Bmat}
\ee
where $b_l(m)=\sqrt{(l^2-m^2)/(h_lh_{l-1})}$.  The matrix $B(m)$ has 
$(l_B-|m|+1)/2$ or $(l_B-|m|)/2$ positive eigenvalues for $l_B-|m|+1$ even 
or odd, respectively.  To see how $B(m)$ is obtained, we first prove 
the following proposition.

\begin{prop}[Ref.~\cite{Garcia-Siewert82}]
\label{reLambda}
Discrete eigenvalues are zeros of $g_l^m$ as $l\to\infty$.
\end{prop}

\begin{proof}
We define
\[
q_l^m(w)=\frac{1}{2}\int_{-1}^1\frac{p_l^m(\mu)}{w-\mu}(1-\mu^2)^{|m|}\,d\mu,
\quad w\notin[-1,1].
\]
For $\nu\notin[-1,1]$, the three-term recurrence relation of $p_l^m$ implies
\bea
\sqrt{(l+1)^2-m^2}q_{l+1}^m(\nu)
&=&
(2l+1)\nu q_l^m(\nu)-\sqrt{l^2-m^2}q_{l-1}^m(\nu)
\nonumber \\
&-&
\left(\mathop{\rm sgn}(m)\right)^m\frac{\sqrt{(2|m|)!}}{(2|m|-1)!!}
\delta_{l,|m|}.
\label{prop:qrecurr}
\eea
By subtracting (\ref{chandra1}) multiplied by $q_l^m(\nu)$ on both sides 
from (\ref{prop:qrecurr}) multiplied by $g_l^m(\nu)$ on both sides, we obtain
\ba
\fl
\sqrt{(l+1)^2-m^2}
\left(q_{l+1}^m(\nu)g_l^m(\nu)-q_l^m(\nu)g_{l+1}^m(\nu)\right)
=(2l+1)\nu q_l^m(\nu)g_l^m(\nu)-\nu h_lq_l^m(\nu)g_l^m(\nu)
\\
-\sqrt{l^2-m^2}\left(q_{l-1}^m(\nu)g_l^m(\nu)-q_l^m(\nu)g_{l-1}^m(\nu)\right)
-\delta_{l,|m|}.
\ea
Suppose $l_B\ge L$.  By taking the summation $\sum_{l=|m|}^{l_B}$ we obtain
\ba
\fl
\sqrt{(l_B+1)^2-m^2}\left[q_{l_B+1}^m(\nu)g_{l_B}^m(\nu)
-q_{l_B}^m(\nu)g_{l_B+1}^m(\nu)\right]
\\
=\sum_{l=|m|}^{l_B}\left((2l+1)\nu-\nu h_l\right)q_l^m(\nu)g_l^m(\nu)-1.
\ea
Noting that $\Lambda^m(\nu)=1-\varpi\nu\sum_{l=|m|}^L
\beta_lg_l^m(\nu)q_l^m(\nu)$, we obtain (the Christoffel-Darboux formula)
\be
\Lambda^m(\nu)=\sqrt{(l_B+1)^2-m^2}\left[q_{l_B}^m(\nu)g_{l_B+1}^m(\nu)
-q_{l_B+1}^m(\nu)g_{l_B}^m(\nu)\right].
\label{prop:A}
\ee
Next we subtract (\ref{prop:qrecurr}) multiplied by $p_l^m(\nu)$ on both sides 
from (\ref{precurr}) multiplied by $q_l^m(\nu)$ on both sides. By summing the 
resulting expression over $l$ from $|m|$ to $l_B$, we have
\be
1=\sqrt{(l_B+1)^2-m^2}\left(p_{l_B+1}^m(\mu)q_{l_B}^m(\nu)
-p_{l_B}^m(\nu)q_{l_B+1}^m(\nu)\right).
\label{prop:B}
\ee
Similarly we subtract (\ref{chandra1}) multiplied by $p_l^m(\nu)$ on both 
sides from (\ref{precurr}) multiplied by $g_l^m(\nu)$ on both sides, and 
take the sum over $l$ from $|m|$ to $l_B$. We obtain
\be
\varpi\nu g^m(\nu,\nu)
=\sqrt{(l_B+1)^2-m^2}\left(p_{l_B+1}^m(\mu)g_{l_B}^m(\nu)
-p_{l_B}^m(\nu)g_{l_B+1}^m(\nu)\right).
\label{prop:C}
\ee
Using (\ref{prop:A}), (\ref{prop:B}), and (\ref{prop:C}), we obtain
\ba
\fl
p_{l_B+1}^m(\nu)\Lambda^m(\nu)
=\sqrt{(l_B+1)^2-m^2}\left[p_{l_B+1}^m(\nu)q_{l_B}^m(\nu)g_{l_B+1}^m(\nu)
-p_{l_B+1}^m(\nu)q_{l_B+1}^m(\nu)g_{l_B}^m(\nu)\right]
\\
=g_{l_B+1}^m(\nu)+\sqrt{(l_B+1)^2-m^2}\left[p_{l_B}^m(\nu)g_{l_B+1}^m(\nu)
-p_{l_B+1}^m(\nu)g_{l_B}^m(\nu)\right]q_{l_B+1}^m(\nu)
\\
=g_{l_B+1}^m(\nu)-\varpi\nu g^m(\nu,\nu)q_{l_B+1}^m(\nu).
\ea
We note that $\lim_{l\to\infty}q_l^m(w)/p_l^m(w)
=\lim_{l\to\infty}Q_l^m(w)/P_l^m(w)=0$ ($w\notin[-1,1]$), where 
$Q_l^m$ is the associated Legendre polynomial of the second kind.  Therefore 
we obtain
\[
\Lambda^m(\nu)=\lim_{l_B\to\infty}\frac{g_{l_B+1}^m(\nu)}{p_{l_B+1}^m(\nu)}.
\]
Thus the proof is completed.
\end{proof}

Let us recall that the recurrence relation
(\ref{chandra1}) for $g_l^m(\nu)$ is derived for an eigenvalue $\nu$ in 
(\ref{separated1d}) and rewrite (\ref{chandra1}) as
\[
\sqrt{\frac{l^2-m^2}{h_lh_{l-1}}}\sqrt{h_{l-1}}g_{l-1}^m(\nu)
+\sqrt{\frac{(l+1)^2-m^2}{h_lh_{l+1}}}\sqrt{h_{l+1}}g_{l+1}^m(\nu)
=\nu\sqrt{h_l}g_l^m(\nu).
\]
Hence eigenvalues of $B(m)$ are zeros of $g_{l_B+1}^m$.  Together with 
Proposition \ref{reLambda}, we see that discrete eigenvalues $\nu_j^m$ can be 
computed as eigenvalues of $B(m)$ for sufficiently large $l_B$.  More 
sophisticated ways of obtaining discrete eigenvalues are discussed in 
Ref.~\cite{Garcia-Siewert89}.  

The tridiagonal matrix $B(m)$ can be alternatively obtained as follows. 
Let us write $\Phi_{\nu}^m(\uv)$ as
\[
\Phi_{\nu}^m(\uv)
=\sum_{l'=0}^{\infty}\sum_{m'=-l'}^{l'}c_{l'm'}^m(\nu)Y_{l'm'}(\uv),
\]
where $c_{l'm'}^m(\nu)\in\Cm$.  Then (\ref{Phieq1}) can be rewritten as
\be
\left(1-\frac{\mu}{\nu}\right)\sum_{l'm'}c_{l'm'}Y_{l'm'}(\uv)
=\varpi\sum_{l'm'}\frac{\beta_{l'}}{2l'+1}c_{l'm'}Y_{l'm'}(\uv).
\label{rePhieq}
\ee
Thus for $|m|\le L$ we have
\[
c_{lm}-\frac{1}{\nu}\sum_{l'm'}c_{l'm'}\int_{\Sm^2}\mu Y_{l'm'}(\uv)
Y_{lm}^*(\uv)\,d\uv
=\frac{\varpi\beta_l}{2l+1}\chi_{[0,L]}(l)c_{lm}.
\]
Using the orthogonality relation for associated Legendre polynomials: 
$\int_{-1}^1P_l^m(\mu)P_{l'}^m\,d\mu=\delta_{ll'}2(l+m)!/[(2l+1)(l-m)!]$, 
we obtain
\ba
\fl
\sqrt{\frac{(2l+1)(2l'+1)}{h_lh_{l'}}}\left(
\sqrt{\frac{l^2-m^2}{4l^2-1}}\delta_{l',l-1}+
\sqrt{\frac{(l+1)^2-m^2}{4(l+1)^2-1}}\delta_{l',l+1}\right)
c_{l'm}\sqrt{\frac{h_{l'}}{2l'+1}}
\\
=\left(b_l(m)\delta_{l',l-1}+b_{l'}(m)\delta_{l',l+1}\right)
c_{l'm}\sqrt{\frac{h_{l'}}{2l'+1}}
=\nu c_{lm}\sqrt{\frac{h_l}{2l+1}}.
\ea
The above equation forms an eigenvalue problem for $B(m)$, and 
$c_{lm}$ are given in terms of eigenvectors of $B(m)$.  

\subsection{Singular eigenfunctions for three dimensions}
\label{pre:case3d}

\begin{defn}[Rotated reference frames]
\label{rrf}
Let $\uvk\in\Cm^3$ be a unit vector such that $\uvk\cdot\uvk=1$.   
We define an invertible linear operator 
$\rrf{\uvk}:\Cm\mapsto\Cm$.  For a function $f_1(\uv)\in\Cm$ ($\uv\in\Sm^2$), 
$\rrf{\uvk}f_1(\uv)$ is the value of $f_1(\uv)$ where $\uv$ is measured in 
the rotated reference frame whose $z$-axis lies in the direction of $\uvk$.
\end{defn}

Suppose that $f_1(\uv)\in\Cm$ is given by spherical harmonics:
\[
f_1(\uv)=\sum_{l=0}^{\infty}\sum_{m=-l}^lf_{lm}Y_{lm}(\uv),
\]
where $f_{lm}\in\Cm$.  Then we have \cite{Dede64,Kobayashi77,Markel04}
\ba
\rrf{\uvk}f_1(\uv)
&=&\sum_{l=0}^{\infty}\sum_{m=-l}^lf_{lm}\sum_{m'=-l}^l
D_{m'm}^l(\va_{\uvk},\theta_{\uvk},0)Y_{lm'}(\uv)
\\
&=&\sum_{l=0}^{\infty}\sum_{m=-l}^lf_{lm}\sum_{m'=-l}^l
e^{-im'\va_{\uvk}}d_{m'm}^l(\theta_{\uvk})Y_{lm'}(\uv),
\ea
where $\theta_{\uvk}$ and $\va_{\uvk}$ are the polar and azimuthal angles 
of $\uvk$ in the laboratory frame.  Here $D_{m'm}^l$ and $d_{m'm}^l$ are 
Wigner's $D$-matrices and $d$-matrices.  Moreover we obtain
\ba
\irrf{\uvk}f_1(\uv)
&=&\sum_{l=0}^{\infty}\sum_{m=-l}^lf_{lm}\sum_{m'=-l}^l
D_{m'm}^l(0,-\theta_{\uvk},-\va_{\uvk})Y_{lm'}(\uv)
\\
&=&\sum_{l=0}^{\infty}\sum_{m=-l}^lf_{lm}\sum_{m'=-l}^l
e^{im\va_{\uvk}}d_{mm'}^l(\theta_{\uvk})Y_{lm'}(\uv).
\ea
We can directly show $\irrf{\uvk}\rrf{\uvk}f_1(\uv)=f_1(\uv)$ by using 
$\sum_{m'=-l}^ld_{m'm}^l(\theta_{\uvk})d_{m'm''}^l(\theta_{\uvk})
=\delta_{mm''}$.
We have for $f_1(\uv),f_2(\uv)\in\Cm$,
\[
\fl
\rrf{\uvk}f_1(\uv)f_2(\uv)
=\left(\rrf{\uvk}f_1(\uv)\right)\left(\rrf{\uvk}f_1(\uv)\right),\qquad
\int_{\Sm^2}\rrf{\uvk}f_1(\uv)\,d\uv=\int_{\Sm^2}f_1(\uv)\,d\uv.
\]

\begin{exa}
For any function $f_1(\uv)$ and the unit vector $\hvv{z}$ in the positive 
direction on the $z$-axis, we have $\rrf{\hvv{z}}f_1(\uv)=f_1(\uv)$.
\end{exa}

\begin{exa}
$\rrf{\uvk}\uv\cdot\uv'=\uv\cdot\uv'$ for $\uv,\uv'\in\Sm^2$.
\end{exa}

\begin{exa}
$\rrf{\uvk}\mu=\sqrt{\frac{4\pi}{3}}\rrf{\uvk}Y_{10}(\uv)
=\sum_{m'=-1}^1e^{-im'\va_{\uvk}}d_{m'0}^1(\theta_{\uvk})Y_{1m'}(\uv)
=\uv\cdot\uvk$.
\end{exa}

\begin{defn}[Plane wave decomposition]
\label{planewave}
Complex unit vectors $\uvk(\nu,\vv{q})\in\Cm^3$ ($\nu\in\Rm$, 
$\vv{q}\in\Rm^2$) are given by
\[
\uvk(\nu,\vv{q})=\left(\begin{array}{c}i\nu\vv{q}\\\hat{k}_z(\nu q)
\end{array}\right),
\]
where $q=|\vv{q}|$ and
\[
\hat{k}_z(\nu q)=\sqrt{1+(\nu q)^2}.
\]
\end{defn}

\begin{exa}
For $\nu\in\Rm$, $\vv{q}\in\Rm^2$,  we obtain
\bea
\fl
\rrf{\uvk(-\nu,\vv{q})}\mu
=\hat{k}_z(\nu q)\mu-i\nu q\sqrt{1-\mu^2}\cos(\va-\va_{\vv{q}}),
\label{rotu}
\\
\fl
\irrf{\uvk(-\nu,\vv{q})}\mu
=\sqrt{\frac{4\pi}{3}}\irrf{\uvk(-\nu,\vv{q})}Y_{10}(\uv)
=\hat{k}_z(\nu q)\mu-i|\nu q|\sqrt{1-\mu^2}\cos\va.
\label{irotu}
\eea
\end{exa}

\begin{defn}[Refs.~\cite{Machida10,Panasyuk06}]
We define
\[
\cos[i\tau(\nu q)]=\cos\theta_{\uvk(\nu,\vv{q})},\qquad
\sin[i\tau(|\nu q|)]=\sin\theta_{\uvk(\nu,\vv{q})}.
\]
Since Wigner's $d$-matrices $d_{mm'}^l(\theta)$ are given in terms of 
$\cos\theta$, we also write
\[
d_{mm'}^l[i\tau(\nu q)]=d_{mm'}^l(\theta_{\uvk(\nu,\vv{q})}).
\]
\end{defn}

To compute Wigner's $d$-matrices, we take square roots such that 
$0\le\mathop{\rm arg}(\sqrt{z})<\pi$ for all $z\in\Cm$ 
\cite{Panasyuk06,Machida10}.  We have
\[
\fl
\cos\va_{\uvk(\nu,\vv{q})}
=\frac{\hvv{x}\cdot(i\nu\vv{q})}{\sqrt{(i\nu\vv{q})\cdot(i\nu\vv{q})}}
=\frac{\nu}{|\nu|}\cos\va_{\vv{q}},
\qquad
\sin\va_{\uvk(\nu,\vv{q})}
=\frac{\hvv{y}\cdot(i\nu\vv{q})}{\sqrt{(i\nu\vv{q})\cdot(i\nu\vv{q})}}
=\frac{\nu}{|\nu|}\sin\va_{\vv{q}},
\]
and
\[
\fl
\cos\theta_{\uvk(\nu,\vv{q})}
=\hvv{z}\cdot\uvk=\hat{k}_z(\nu q),
\qquad
\sin\theta_{\uvk(\nu,\vv{q})}
=\sqrt{1-\cos^2\theta_{\uvk(\nu,\vv{q})}}=\sqrt{(i\nu\vv{q})\cdot(i\nu\vv{q})}
=i|\nu q|.
\]
In particular we obtain
\[
\va_{\uvk(\nu,\vv{q})}=\cases{
\va_{\vv{q}},&for $\nu>0$,
\\
\va_{\vv{q}}+\pi,&for $\nu<0$,
}
\]
where $\va_{\vv{q}}$ is the polar angle of $\vv{q}$.

Let us consider
\be
\uv\cdot\nabla I(\vv{r},\uv)+I(\vv{r},\uv)
=\varpi\int_{\Sm^2}p(\uv,\uv')I(\vv{r},\uv')\,d\uv',
\label{rte3d}
\ee
where $\vv{r}\in\Rm^3$, $\uv\in\Sm^2$.  Solutions to the above 
equation are given by a superposition of eigenmodes
\be
I(\vv{r},\uv)=\rrf{\uvk}\Phi_{\nu}^m(\uv)e^{-\uvk\cdot\vv{r}/\nu},
\label{separated3d}
\ee
where $\uvk=\uvk(\nu,\vv{q})$.  To see this we substitute the separated 
solution (\ref{separated3d}) in the above homogeneous three-dimensional 
radiative transport equation (\ref{rte3d}) and obtain
\be
\left(1-\frac{\rrf{\uvk}\mu}{\nu}\right)\rrf{\uvk}\Phi_{\nu}^m(\uv)
=\varpi\int_{\Sm^2}p(\uv,\uv')\rrf{\uvk}\Phi_{\nu}^m(\uv')\,d\uv'.
\label{Phieq3D1}
\ee
The right-hand side can be written as
\be
\varpi\int_{\Sm^2}p(\uv,\uv')\rrf{\uvk}\Phi_{\nu}^m(\uv')\,d\uv'
=\varpi\int_{\Sm^2}p(\rrf{\uvk}\uv,\rrf{\uvk}\uv')
\rrf{\uvk}\Phi_{\nu}^m(\uv')\,d\uv'.
\label{Phieq3D2}
\ee
That is,
\[
\rrf{\uvk}\left(1-\frac{\mu}{\nu}\right)\Phi_{\nu}^m(\uv)
=\rrf{\uvk}\varpi\int_{\Sm^2}p(\uv,\uv')\Phi_{\nu}^m(\uv')\,d\uv'.
\]
Thus the three-dimensional equation (\ref{Phieq3D1}) reduces to the 
one-dimensional equation (\ref{Phieq1}).  Recall that $\Phi_{\nu}^m(\uv)$ 
given in (\ref{Phidef}) is constructed so that (\ref{Phieq1}) obtained from 
(\ref{rte1d}) and (\ref{separated1d}) is satisfied. We have
\be
\rrf{\uvk}\phi^m(\nu,\mu)
=\frac{\varpi\nu}{2}\mathcal{P}\frac{g^m(\nu,\uv\cdot\uvk)}{\nu-\uv\cdot\uvk}
+\lambda^m(\nu)\left(1-\nu^2\right)^{-|m|}\delta(\nu-\uv\cdot\uvk).
\label{rotatedphi}
\ee

\begin{prop}
\label{singularfunc}
The following orthogonality relation holds.
\[
\int_{\Sm^2}\mu\left(\rrf{\uvk(\nu,\vv{q})}\Phi_{\nu}^m(\uv)\right)
\left(\rrf{\uvk(\nu',\vv{q})}\Phi_{\nu'}^{m'*}(\uv)\right)\,d\uv
=2\pi\hat{k}_z(\nu q)\mathcal{N}(\nu)\delta_{\nu\nu'}\delta_{mm'}.
\]
\end{prop}

\begin{proof}
The full-range orthogonality is obtained in \cite{Machida14} 
through the Green's function.  Here we give a direct proof.

We perform separation of variables to the homogeneous equation by assuming 
the form (\ref{separated3d}).  By substituting the separated solution into 
the radiative transport equation (\ref{rte3d}), we obtain
\[
\fl
\left(1-\frac{\rrf{\uvk}\mu}{\nu}\right)\rrf{\uvk}\Phi_{\nu}^m(\uv)
=\varpi\sum_{l=0}^L\sum_{m=-l}^l\frac{\beta_l}{2l+1}Y_{lm}(\uv)
\int_{\Sm^2}Y_{lm}^*(\uv')\rrf{\uvk}\Phi_{\nu}^m(\uv')\,d\uv',
\]
For fixed $\vv{q}$, we consider $(m_1,\nu_1^{m_1})$ and $(m_2,\nu_2^{m_2})$.  
Let us write $\uvk_1=\uvk(\nu_1^{m_1},\vv{q})$, 
$\uvk_2=\uvk(\nu_2^{m_2},\vv{q})$.  We write the following two equations.
\ba
\fl
\left(\rrf{\uvk_2}\Phi_{\nu_2}^{m_2}(\uv)\right)\rrf{\uvk_1}
\left(1-\frac{\mu}{\nu_1}\right)\Phi_{\nu_1}^{m_1}(\uv)
\\
=\varpi\sum_{l=0}^L\sum_{m=-l}^l\frac{\beta_l}{2l+1}
Y_{lm}(\uv)\left(\rrf{\uvk_2}\Phi_{\nu_2}^{m_2}(\uv)\right)
\int_{\Sm^2}Y_{lm}^*(\uv')\left(\rrf{\uvk_1}\Phi_{\nu_1}^{m_1}(\uv')
\right)\,d\uv',
\\
\fl
\left(\rrf{\uvk_1}\Phi_{\nu_1}^{m_1}(\uv)\right)\rrf{\uvk_2}
\left(1-\frac{\mu}{\nu_2}\right)\Phi_{\nu_2}^{m_2}(\uv)
\\
=\varpi\sum_{l=0}^L\sum_{m=-l}^l\frac{\beta_l}{2l+1}
Y_{lm}^*(\uv)\left(\rrf{\uvk_1}\Phi_{\nu_1}^{m_1}(\uv)\right)
\int_{\Sm^2}Y_{lm}(\uv')\left(\rrf{\uvk_2}\Phi_{\nu_2}^{m_2}(\uv')
\right)\,d\uv'.
\ea
We note (\ref{rotu}).  By subtraction and integration over $\uv$ we have
\ba
\int_{\Sm^2}
\left(\frac{\rrf{\uvk_2}\mu}{\nu_2}-\frac{\rrf{\uvk_1}\mu}{\nu_1}\right)
\left(\rrf{\uvk_1}\Phi_{\nu_1}^{m_1}(\uv)\right)
\left(\rrf{\uvk_2}\Phi_{\nu_2}^{m_2}(\uv)\right)\,d\uv
\\
=
\left(\frac{\hat{k}_z(\nu_2q)}{\nu_2}-\frac{\hat{k}_z(\nu_1q)}{\nu_1}\right)
\int_{\Sm^2}\mu\left(\rrf{\uvk_1}\Phi_{\nu_1}^{m_1}(\uv)\right)
\left(\rrf{\uvk_2}\Phi_{\nu_2}^{m_2}(\uv)\right)\,d\uv
\\
=0.
\ea
Therefore,
\be
\int_{\Sm^2}\mu\left(\rrf{\uvk_1}\Phi_{\nu_1}^{m_1}(\uv)\right)
\left(\rrf{\uvk_2}\Phi_{\nu_2}^{m_2}(\uv)\right)\,d\uv=0
\quad\mbox{for}\;\nu_1\neq\nu_2.
\label{ortho1}
\ee
Suppose $\nu=\nu_1=\nu_2$, $\uvk=\uvk_1=\uvk_2$, $m_1\neq m_2$.  
In this case we have
\bea
\fl
\int_{\Sm^2}\mu\left(\rrf{\uvk}\Phi_{\nu}^{m_1}(\uv)\right)
\left(\rrf{\uvk}\Phi_{\nu}^{m_2}(\uv)\right)\,d\uv
=\int_{\Sm^2}\mu\Phi_{\nu}^{m_1}(\uv)\Phi_{\nu}^{m_2}(\uv)\,d\uv
\nonumber \\
=\int_0^{2\pi}e^{i(m_1+m_2)\va}\,d\va\;
\int_{-1}^1\mu\phi^{m_1}(\nu,\mu)\phi^{m_2}(\nu,\mu)
\left(1-\mu^2\right)^{(|m_1|+|m_2|)/2}\,d\mu
\nonumber \\
\propto
\delta_{m_1,-m_2}.
\label{ortho2}
\eea
We note that
\[
\Phi_{\nu}^{-m}(\uv)=\Phi_{\nu}^m(\uv)^*.
\]
Using (\ref{ortho1}) and (\ref{ortho2}), for arbitrary $\nu,\nu',m,m'$, we have
\[
\int_{\Sm^2}\mu\left(\rrf{\uvk}\Phi_{\nu}^m(\uv)\right)
\left(\rrf{\uvk'}\Phi_{\nu'}^{m'}(\uv)^*\right)\,d\uv
\propto\delta_{\nu\nu'}\delta_{mm'}.
\]

If $\nu=\nu'$, $m=m'$, we have
\ba
\int_{\Sm^2}\mu\left(\rrf{\uvk}\Phi_{\nu}^m(\uv)\right)
\left(\rrf{\uvk}\Phi_{\nu}^m(\uv)^*\right)\,d\uv
&=&
\int_{\Sm^2}\left(\irrf{\uvk}\mu\right)\Phi_{\nu}^m(\uv)
\Phi_{\nu}^m(\uv)^*\,d\uv
\\
&=&
\int_{\Sm^2}\left(\irrf{\uvk}\mu\right)\left[\phi^m(\nu,\mu)\right]^2
\left(1-\mu^2\right)^{|m|}\,d\uv.
\ea
Hence,
\ba
\int_{\Sm^2}\mu\left(\rrf{\uvk}\Phi_{\nu}^m(\uv)\right)
\left(\rrf{\uvk}\Phi_{\nu}^m(\uv)^*\right)\,d\uv
&=&
2\pi\hat{k}_z(\nu q)\int_{-1}^1\mu\left[\phi^m(\nu,\mu)\right]^2
\left(1-\mu^2\right)^{|m|}\,d\mu
\\
&=&
2\pi\hat{k}_z(\nu q)\mathcal{N}^m(\nu),
\ea
where the normalization factor $\mathcal{N}^m(\nu)$ is given in 
(\ref{normalizationN}).  Thus we obtain the full-range orthogonality relation.
\end{proof}

\subsection{Method of rotated reference frames}
\label{pre:mrrf}

The method of rotated reference frames does not rely on singular 
eigenfunctions $\Phi_{\nu}^{m'}(\uv)$ and uses the expansion 
(\ref{mrrfexpansion}), in which $c_{lm}^{m'}(\nu)$ are unknown coefficients 
that can be fully numerically computed as eigenvectors of $B(m')$. The method 
is summarized 
in \ref{mrrf}. We describe below how the matrix $B(m')$ appears in 
this method.

We plug (\ref{mrrfexpansion}) into (\ref{Phieq3D1}):
\[
\fl
\left(1-\frac{\rrf{\uvk}\mu}{\nu}\right)
\sum_{lm}c_{lm}^{m'}\rrf{\uvk}Y_{lm}(\uv)
=\varpi\int_{\Sm^2}p(\rrf{\uvk}\uv,\rrf{\uvk}\uv')
\sum_{lm}c_{lm}^{m'}\rrf{\uvk}Y_{lm}(\uv')\,d\uv'.
\]
By operating $\irrf{\uvk}$, the above equation reduces to (\ref{rePhieq}), 
from which the matrix $B(m')$ is derived.

\section{The $F_N$ method in three dimensions}
\label{fn}

To show how the $F_N$ method can be extended to three dimensions, 
we will consider the half-space geometry in which a homogeneous random 
medium with optical parameter $\varpi$ exists only in the lower half $z<0$. 
By the Placzek lemma \cite{Placzek} we can consider the following 
radiative transport equation in $\Rm^3$ instead of (\ref{rte}).
\[
\fl
\cases{
\uv\cdot\nabla\psi(\vv{r},\uv)+\psi(\vv{r},\uv)
=\varpi\int_{\Sm^2}p(\uv,\uv')\psi(\vv{r},\uv')\,d\uv'
+\chi_{(0,\infty)}(z)S(\vv{r},\uv)+\mu I(\vv{r},\uv)\delta(z),
&$z\in(-\infty,\infty)$,
\\
\psi(\vv{r},\uv)\to0,
&$|z|\to\infty$,
}
\]
where $\chi_{(0,\infty)}(z)=1$ for $z>0$ and $=0$ otherwise.  
We have the jump condition
\[
\psi(\vv{\rho},0^+,\uv)-\psi(\vv{\rho},0^-,\uv)=I(\vv{\rho},0,\uv).
\]
Since $I(\vv{\rho},0,\uv)$ is given by only eigenmodes with positive 
eigenvalues and $\psi(\vv{\rho},0^-,\uv)$ is given by only eigenmodes 
with negative eigenvalues, we see that $\psi(\vv{\rho},0^-,\uv)=0$.  Therefore 
we obtain the relation
\[
\psi(\vv{r},\uv)=\cases{
I(\vv{r},\uv),&$z>0$,
\\
0,&$z<0$.
}
\]
Let us introduce the Green's function $G(\vv{r},\uv;\vv{r}_0,\uv_0)$ for 
the infinite medium as
\[
\fl
\cases{
\uv\cdot\nabla G(\vv{r},\uv)+G(\vv{r},\uv)
=\varpi\int_{\Sm^2}p(\uv,\uv')G(\vv{r},\uv')\,d\uv'
+\delta(\vv{r}-\vv{r}_0)\delta(\uv-\uv_0),
&$z\in(-\infty,\infty)$,
\\
G(\vv{r},\uv)\to0,
&$|z|\to\infty$.
}
\]
Thus we obtain
\ba
\psi(\vv{r},\uv)
&=&
\int_{\Sm^2}\int_{\Rm^2}G(\vv{r},\uv;\vv{\rho}',0,\uv')
\mu'I(\vv{\rho}',0,\uv')\,d\vv{\rho}'d\uv'
\\
&+&
\int_{\Sm^2}\int_0^{\infty}\int_{\Rm^2}
G(\vv{r},\uv;\vv{\rho}',z',\uv')S(\vv{\rho}',z',\uv')\,d\vv{\rho}'dz'd\uv',
\quad\vv{r}\in\Rm^3,\;\uv\in\Sm^2.
\ea
The Green's function is obtained as \cite{Machida14}
\[
G(\vv{r},\uv;\vv{r}_0,\uv_0)
=
\frac{1}{(2\pi)^2}\int_{\Rm^2}e^{-i\vv{q}\cdot(\vv{\rho}-\vv{\rho}_0)}
\tilde{G}(\vv{q};z,\uv;z_0,\uv_0)\,d\vv{q}.
\]
Here,
\ba
\fl
\tilde{G}(\vv{q};z,\uv;z_0,\uv_0)
=\sum_{m=-L}^L\Biggl\{\sum_{j=0}^{M^m-1}
\frac{1}{2\pi\hat{k}_z(\nu_j^mq)\mathcal{N}(\nu_j^m)}
\rrf{\uvk(\pm\nu_j^m,\vv{q})}\Phi_{j\pm}^m(\uv)\Phi_{j\pm}^{m*}(\uv_0)
e^{-\hat{k}_z(\nu_j^mq)|z-z_0|/\nu_j^m}
\\
+
\int_0^1\frac{1}{2\pi\hat{k}_z(\nu q)\mathcal{N}(\nu)}\rrf{\uvk(\pm\nu,\vv{q})}
\Phi_{\pm\nu}^m(\uv)\Phi_{\pm\nu}^{m*}(\uv_0)e^{-\hat{k}_z(\nu q)|z-z_0|/\nu}
\,d\nu\Biggr\},
\ea
where upper signs are chosen for $z>z_0$ and lower signs are chosen for 
$z<z_0$.  By letting $z\to0^+$ we obtain
\ba
I(\vv{\rho},0,\uv)
&=&
\int_{\Sm^2}\int_{\Rm^2}G(\vv{\rho},0^+,\uv;\vv{\rho}',0,\uv')
\mu'I(\vv{\rho}',0,\uv')\,d\vv{\rho}'d\uv'
\\
&+&
\int_{\Sm^2}\int_0^{\infty}\int_{\Rm^2}
G(\vv{\rho},0,\uv;\vv{\rho}',z',\uv')S(\vv{\rho}',z',\uv')\,
d\vv{\rho}'dz'd\uv',
\ea
where $\uv\in\Sm^2$.  We have
\bea
\tilde{I}(\vv{q},0,\uv)
&=&
\int_{\Sm^2}\tilde{G}(\vv{q};0^+,\uv;0,\uv')\mu'
\tilde{I}(\vv{q},0,\uv')\,d\uv'
\nonumber \\
&+&
\int_{\Sm^2}\int_0^{\infty}
\tilde{G}(\vv{q};0,\uv;z',\uv')\tilde{S}(\vv{q},z',\uv')\,dz'd\uv'.
\label{SBeq7}
\eea

\begin{defn}
Let $\xi^m$ denote the positive eigenvalues, i.e., 
$\xi^m=\nu_j^m$ ($j=0,1,\dots,M^m-1$) or $\xi^m=\nu\in(0,1)$.  We drop the 
superscript and write $\xi=\xi^m$ if there is no confusion.
\end{defn}

If we multiply (\ref{SBeq7}) by 
$\mu\rrf{\uvk(-\xi,\vv{q})}\Phi_{-\xi}^{m'*}(\uv)$ with some $m'$ 
and $\xi=\xi^{m'}>0$, and integrate over $\Sm^2$, we obtain
\ba
\fl
\int_{\Sm^2}\mu\left(\rrf{\uvk(-\xi,\vv{q})}\Phi_{-\xi}^{m'*}(\uv)
\right)\tilde{I}(\vv{q},0,\uv)\,d\uv
\\
=\int_{\Sm^2}\int_0^{\infty}\left(\rrf{\uvk(-\xi,\vv{q})}
\Phi_{-\xi}^{m'*}(\uv')\right)e^{-\hat{k}_z(\xi q)z'/\xi}
\tilde{S}(\vv{q},z',\uv')\,dz'd\uv'.
\ea
Hence we can write the above equation as
\bea
\fl
\int_{\Sm^2_+}\mu\left(\rrf{\uvk(-\xi,\vv{q})}\Phi_{-\xi}^{m'*}(-\uv)
\right)\tilde{I}(\vv{q},0,-\uv)\,d\uv
=\int_{\Sm^2_+}\mu\left(\rrf{\uvk(-\xi,\vv{q})}\Phi_{-\xi}^{m'*}(\uv)
\right)\tilde{f}(\vv{q},\uv)\,d\uv
\nonumber \\
-\int_{\Sm^2}\int_0^{\infty}\left(\rrf{\uvk(-\xi,\vv{q})}
\Phi_{-\xi}^{m'*}(\uv')\right)e^{-\hat{k}_z(\xi q)z'/\xi}
\tilde{S}(\vv{q},z',\uv')\,dz'd\uv'.
\label{projection1}
\eea
By the expansion in (\ref{FNexpansion}), we obtain the following 
key $F_N$ equation
\be
\sum_{m=-l_{\rm max}}^{l_{\rm max}}\sum_{l=|m|,|m|+2,\dots}
A_{lm}^{m'}(\xi,\vv{q})c_{lm}(\vv{q})
=K^{m'}(\xi,\vv{q}),
\label{keyFN}
\ee
where $-L\le m'\le L$.  Here,
\ba
A_{lm}^{m'}(\xi,\vv{q})
&=&
\int_{\Sm^2_+}\mu Y_{lm}(\uv)
\rrf{\uvk(-\xi,\vv{q})}\Phi_{-\xi}^{m'*}(-\uv)\,d\uv,
\\
K^{m'}(\xi,\vv{q})
&=&
\int_{\Sm^2_+}\mu\tilde{f}(\vv{q},\uv)\rrf{\uvk(-\xi,\vv{q})}
\Phi_{-\xi}^{m'*}(\uv)\,d\uv
\\
&-&
\int_{\Sm^2}\int_0^{\infty}
e^{-\hat{k}_z(\xi q)z'/\xi}\tilde{S}(\vv{q},z',\uv')
\rrf{\uvk(-\xi,\vv{q})}\Phi_{-\xi}^{m'*}(\uv')\,dz'd\uv'.
\ea

\begin{rmk}
In the above proof we used the Green's function in the free space to derive 
(\ref{projection1}). This approach is similar to the $C_N$ method 
\cite{Benoist-Kavenoky68,Kavenoky78}. If the Green's function for the half 
space is used, we can explicitly give $\tilde{I}(\vv{q},0,-\uv)$ without 
relying on (\ref{FNexpansion}) and (\ref{keyFN}) \cite{Siewert-Benoist79}. 
However, the half-space Green's function in three dimensions is not yet known.
\end{rmk}

We obtain
\[
A_{lm}^{m'}(\xi,\vv{q})=A_{lm}^{m'}(\xi,q)e^{im\va_{\vv{q}}},
\]
where
\bea
\fl
A_{lm}^{m'}(\xi,q)
=(-1)^m\hat{k}_z(\xi q)\sqrt{\frac{\pi}{2l+1}}d_{mm'}^l[i\tau(\xi q)]
\left(\sqrt{(l+1)^2-{m'}^2}g_{l+1}^{m'}(\xi)
+\sqrt{l^2-{m'}^2}g_{l-1}^{m'}(\xi)\right)
\nonumber \\
\fl
-i\frac{|\xi q|}{2}\sqrt{\frac{\pi}{2l+1}}(-1)^m
\sum_{m''=-l}^ld_{mm''}^l[i\tau(\xi q)]
\nonumber \\
\fl
\times\Biggl[\delta_{m'',m'-1}\left(\sqrt{(l-m'')(l-m')}g_{l-1}^{m'}(\xi)
-\sqrt{(l+m'+1)(l+m')}g_{l+1}^{m'}(\xi)\right)
\nonumber \\
\fl
+\delta_{m'',m'+1}\left(\sqrt{(l-m'+1)(l-m')}g_{l+1}^{m'}(\xi)
-\sqrt{(l+m'')(l+m')}g_{l-1}^{m'}(\xi)\right)\Biggr]
\nonumber \\
\fl
+\frac{\varpi\xi}{2}(-1)^l\sqrt{\frac{2l+1}{4\pi}\frac{(l-m)!}{(l+m)!}}
[\mathop{\rm sgn}(m')]^{m'}\frac{\sqrt{(2|m'|)!}}{(2|m'|-1)!!}
\sum_{m''=-|m'|}^{|m'|}(-1)^{m''}\sqrt{\frac{(|m'|-m'')!}{(|m'|+m'')!}}
\nonumber \\
\fl
\times d_{m'',-m'}^{|m'|}[i\tau(\xi q)]
\int_{\Sm^2_+}
\frac{g^{m'}\left(-\xi,\hat{k}_z(\xi q)\mu-i\xi q\sqrt{1-\mu^2}\cos\va\right)}
{\xi+\hat{k}_z(\xi q)\mu-i\xi q\sqrt{1-\mu^2}\cos\va}
\mu P_{|m'|}^{m''}(\mu)P_l^m(\mu)e^{i(m+m'')\va}\,d\uv.
\nonumber \\
\label{Aq}
\eea
If $K^{m'}(\xi,\vv{q})$ is independent of $\va_{\vv{q}}$ and 
$K^{-m'}=K^{m'}$, then
\[
c_{lm}(\vv{q})=c_{lm}(q)e^{-im\va_{\vv{q}}},\qquad
c_{l,-m}(q)=(-1)^mc_{lm}(q).
\]
Here the coefficients $c_{lm}(q)$ are solutions to
\bea
\fl
\sum_{m=0}^{l_{\rm max}}
\sum_{\alpha=0}^{\left\lfloor(l_{\rm max}-m)/2\right\rfloor}
\left[A_{m+2\alpha,m}^{m'}(\xi,q)+(1-\delta_{m0})(-1)^m
A_{m+2\alpha,-m}^{m'}(\xi,q)\right]c_{m+2\alpha,m}(q)
\nonumber \\
=K^{m'}(\xi,\vv{q}),
\label{keyFN2}
\eea
where $A_{m+2\alpha,m}^{m'}(\xi,q)$ are given in (\ref{Aq}).

The rest of the section is devoted to the calculations of (\ref{Aq}) and 
(\ref{keyFN2}).

First, $A_{lm}^{m'}(\xi,\vv{q})$ are computed as follow.  
We begin by noting that
\bea
\fl
A_{lm}^{m'}(\xi,\vv{q})
=\int_{\Sm^2}\mu Y_{lm}(\uv)\rrf{\uvk(-\xi,\vv{q})}
\Phi_{-\xi}^{m'*}(-\uv)\,d\uv
-\int_{\Sm^2_-}\mu Y_{lm}(\uv)\rrf{\uvk(-\xi,\vv{q})}
\Phi_{-\xi}^{m'*}(-\uv)\,d\uv
\nonumber \\
=\int_{\Sm^2}\left(\irrf{\vv{\uvk}(-\xi,\vv{q})}\mu Y_{lm}(\uv)
\right)\Phi_{-\xi}^{m'*}(-\uv)\,d\uv
+\int_{\Sm^2_+}\mu Y_{lm}(-\uv)\rrf{\uvk(-\xi,\vv{q})}
\Phi_{-\xi}^{m'*}(\uv)\,d\uv.
\nonumber \\
\label{Asplit}
\eea
We obtain the first term on the right-hand side of (\ref{Asplit}) as
\ba
\fl
\mbox{[1st term]}
=\int_{\Sm^2}\Phi_{-\xi}^{m'*}(-\uv)\irrf{\uvk(-\xi,\vv{q})}
\mu Y_{lm}(\uv)\,d\uv
\\
=\sum_{m''=-l}^le^{im\va_{\uvk}}d_{mm''}^l(\theta_{\uvk})
\int_{\Sm^2}\Phi_{-\xi}^{m'*}(-\uv)
\left(\irrf{\uvk(-\xi,\vv{q})}\mu\right)Y_{lm''}(\uv)\,d\uv
\\
=(-1)^l\sum_{m''=-l}^le^{im\va_{\uvk}}d_{mm''}^l(\theta_{\uvk})
\int_{\Sm^2}\Phi_{-\xi}^{m'*}(\uv)\left(-\hat{k}_z(\xi q)\mu+
i|\xi q|\sqrt{1-\mu^2}\cos\va\right)Y_{lm''}(\uv)\,d\uv.
\ea
Here,
\ba
\fl
\int_{\Sm^2}\Phi_{-\xi}^{m'*}(\uv)\mu Y_{lm''}(\uv)\,d\uv
\\
\fl
=\sqrt{\frac{(l+1)^2-{m''}^2}{4(l+1)^2-1}}
\int_{\Sm^2}\Phi_{-\xi}^{m'*}(\uv)Y_{l+1,m''}(\uv)\,d\uv
+\sqrt{\frac{l^2-{m''}^2}{4l^2-1}}
\int_{\Sm^2}\Phi_{-\xi}^{m'*}(\uv)Y_{l-1,m''}(\uv)\,d\uv
\\
\fl
=\delta_{m'm''}(-1)^{m'}\sqrt{\frac{\pi}{2l+1}}\left(
\sqrt{(l+1)^2-{m'}^2}g_{l+1}^{m'}(-\xi)+\sqrt{l^2-{m'}^2}g_{l-1}^{m'}(-\xi)
\right),
\ea
where we used $\mu P_l^{m'}(\mu)=
[(l+m')P_{l-1}^{m'}(\mu)+(l-m'+1)P_{l+1}^{m'}]/(2l+1)$.  We also have
\ba
\fl
\int_{\Sm^2}\Phi_{-\xi}^{m'*}(\uv)
\sqrt{1-\mu^2}\cos\va Y_{lm''}(\uv)\,d\uv
=\sqrt{\frac{(2l+1)\pi}{4}\frac{(l-m'')!}{(l+m'')!}}
\left(\delta_{m'',m'-1}+\delta_{m'',m'+1}\right)
\\
\times
\int_{-1}^1\phi^{m'}(-\xi,\mu)(1-\mu^2)^{(|m'|+1)/2}P_l^{m''}(\mu)\,d\mu.
\ea
Using $\sqrt{1-\mu^2}P_l^{m'-1}(\mu)=
[P_{l-1}^{m'}(\mu)-P_{l+1}^{m'}(\mu)]/(2l+1)$, 
$\sqrt{1-\mu^2}P_l^{m'+1}(\mu)=
(l-m')\mu P_l^{m'}(\mu)-(l+m')P_{l-1}^{m'}(\mu)$, we obtain
\ba
\fl
\int_{\Sm^2}\Phi_{-\xi}^{m'*}(\uv)
\sqrt{1-\mu^2}\cos\va Y_{lm''}(\uv)\,d\uv
\\
\fl
=\frac{1}{2}\sqrt{\frac{\pi}{2l+1}}(-1)^{l+1}\Biggl[
\delta_{m'',m'-1}\left(\sqrt{(l-m'+1)(l-m')}g_{l-1}^{m'}(\xi)
-\sqrt{(l+m'+1)(l+m')}g_{l+1}^{m'}(\xi)\right)
\\
\fl
+\delta_{m'',m'+1}\left(\sqrt{(l-m'+1)(l-m')}g_{l+1}^{m'}(\xi)
-\sqrt{(l+m'+1)(l+m')}g_{l-1}^{m'}(\xi)\right)\Biggr].
\ea
Therefore,
\ba
\fl
\mbox{[1st term]}
\\
\fl
=-(-1)^{l+m'}\hat{k}_z(\xi q)\sqrt{\frac{\pi}{2l+1}}
e^{im\va_{\uvk}}d_{mm'}^l(\theta_{\uvk})\left(
\sqrt{(l+1)^2-{m'}^2}g_{l+1}^{m'}(-\xi)+\sqrt{l^2-{m'}^2}g_{l-1}^{m'}(-\xi)
\right)
\\
\fl
-i\frac{|\xi q|}{2}\sqrt{\frac{\pi}{2l+1}}
\sum_{m''=-l}^le^{im\va_{\uvk}}d_{mm''}^l(\theta_{\uvk})
\\
\fl
\times\Biggl[\delta_{m'',m'-1}\left(
\sqrt{(l-m'+1)(l-m')}g_{l-1}^{m'}(\xi)-\sqrt{(l+m'+1)(l+m')}g_{l+1}^{m'}(\xi)
\right)
\\
\fl
+\delta_{m'',m'+1}\left(\sqrt{(l-m'+1)(l-m')}g_{l+1}^{m'}(\xi)
-\sqrt{(l+m'+1)(l+m')}g_{l-1}^{m'}(\xi)\right)\Biggr].
\ea
We will use
\[
\fl
(1-\mu^2)^{|m'|/2}e^{-im'\va}
=\frac{(-1)^{|m'|}}{(2|m'|-1)!!}P_{|m'|}^{|m'|}(\mu)e^{-im'\va}
=[\mathop{\rm sgn}(m')]^{m'}\frac{\sqrt{4\pi(2|m'|+1)!}}{(2|m'|+1)!!}
Y_{|m'|,-m'}(\uv).
\]
The second term on the right-hand side of (\ref{Asplit}) is calculated as
\ba
\fl
\mbox{[2nd term]}
=\int_{\Sm^2_+}\left(\rrf{\uvk(-\xi,\vv{q})}\Phi_{-\xi}^{m'*}(\uv)
\right)\mu Y_{lm}(-\uv)\,d\uv
\\
=\frac{\varpi\xi}{2}\int_{\Sm^2_+}
\frac{g^{m'}\left(-\xi,\hat{k}_z(\xi q)\mu-i\xi q\sqrt{1-\mu^2}
\cos(\va-\va_{\vv{q}})\right)}{\xi+\hat{k}_z(\xi q)\mu
-i\xi q\sqrt{1-\mu^2}\cos(\va-\va_{\vv{q}})}
[\mathop{\rm sgn}(m')]^{m'}\frac{\sqrt{4\pi(2|m'|+1)!}}{(2|m'|+1)!!}
\\
\times
\sum_{m''=-|m'|}^{|m'|}e^{-im''\va_{\uvk}}d_{m'',-m'}^{|m'|}(\theta_{\uvk})
\mu Y_{|m'|m''}(\uv)Y_{lm}(-\uv)\,d\uv
\\
\fl
=\frac{\varpi\xi}{2}(-1)^l\sqrt{\frac{2l+1}{4\pi}\frac{(l-m)!}{(l+m)!}}
[\mathop{\rm sgn}(m')]^{m'}\frac{\sqrt{(2|m'|)!}}{(2|m'|-1)!!}
\sum_{m''=-|m'|}^{|m'|}\sqrt{\frac{(|m'|-m'')!}{(|m'|+m'')!}}
e^{-im''\va_{\uvk}}d_{m'',-m'}^{|m'|}(\theta_{\uvk})
\\
\times\int_{\Sm^2_+}\left(\rrf{\uvk(-\xi,\vv{q})}
\frac{g^{m'}(-\xi,\mu)}{\xi+\mu}\right)
\mu P_{|m'|}^{m''}(\mu)P_l^m(\mu)e^{i(m+m'')\va}\,d\uv.
\ea
We note that the relation $g^m(-\xi,\mu)=g^m(\xi,-\mu)$ implies 
$\rrf{\uvk}\phi^m(-\xi,\mu)=\rrf{\uvk}\phi^m(\xi,-\mu)$ 
for a fixed $\uvk$.  Thus (\ref{Aq}) is obtained.

Next, (\ref{keyFN2}) is obtained as follows.  
Using $p_l^{-m}(\mu)=(-1)^mp_l^m(\mu)$, $g_l^{-m}(\xi)=(-1)^mg_l^m(\xi)$, 
and $g^m(\xi,\mu)=g^{-m}(\xi,\mu)$, we can show that
\[
A_{l,-m}^{-m'}(\xi,\vv{q})e^{im\va_{\vv{q}}}
=(-1)^mA_{lm}^{m'}(\xi,\vv{q})e^{-im\va_{\vv{q}}}.
\]
Since we assume $K^{-m'}(\xi,\vv{q})=K^{m'}(\xi,\vv{q})$, we have
\ba
\fl
\sum_{m=-l_{\rm max}}^{l_{\rm max}}\sum_lA_{lm}^{m'}(\xi,\vv{q})c_{lm}(\vv{q})
=\sum_{m=-l_{\rm max}}^{l_{\rm max}}\sum_lA_{lm}^{-m'}(\xi,\vv{q})
c_{lm}(\vv{q})
\\
=\sum_{m=-l_{\rm max}}^{l_{\rm max}}\sum_lA_{l,-m}^{-m'}(\xi,\vv{q})
c_{l,-m}(\vv{q})
=\sum_{m=-l_{\rm max}}^{l_{\rm max}}\sum_l
A_{lm}^{m'}(\xi,\vv{q})(-1)^me^{-2im\va_{\vv{q}}}c_{l,-m}(\vv{q}).
\ea
This implies
\[
c_{l,-m}(\vv{q})=(-1)^me^{2im\va_{\vv{q}}}c_{lm}(\vv{q}).
\]
Moreover since we assume that $K^{m'}(\xi,\vv{q})$ is independent of 
$\va_{\vv{q}}$, we have
\[
\sum_{m=-l_{\rm max}}^{l_{\rm max}}\sum_lA_{lm}^{m'}(\xi,\vv{q})c_{lm}(\vv{q})
=\sum_{m=-l_{\rm max}}^{l_{\rm max}}\sum_l
A_{lm}^{m'}(\xi,q)e^{im\va_{\vv{q}}}c_{lm}(\vv{q}).
\]
This implies
\[
c_{lm}(\vv{q})=c_{lm}(q)e^{-im\va_{\vv{q}}}.
\]
Therefore we obtain
\[
c_{l,-m}(q)=(-1)^mc_{lm}(q).
\]
By using this relation in (\ref{keyFN}), we obtain (\ref{keyFN2}).

\section{Structured illumination}
\label{numerics}

Let us consider a structured illumination in the half space:
\[
\fl
\cases{
\uv\cdot\nabla I(\vv{r},\uv)+I(\vv{r},\uv)
=\varpi\int_{\Sm^2}p(\uv,\uv')I(\vv{r},\uv')\,d\uv',
&$z>0$,
\\
I(\vv{r},\uv)=f(\vv{\rho},\uv),
&$z=0$, $\mu\in(0,1]$,
\\
I(\vv{r},\uv)\to0,
&$z\to\infty$.
}
\]
Here the incoming boundary value $f$ is given by
\[
f(\vv{\rho},\uv)
=I_0\left[1+A_0\cos(\vv{q}_0\cdot\vv{\rho}+B_0)\right]\delta(\uv-\uv_0),\qquad
\uv_0\in\Sm^2_+,
\]
where $I_0$ is the amplitude, $A_0$ is the modulation depth, and $B_0$ is 
the phase of the source.  It is enough if we consider \cite{Lukic09}
\be
f(\vv{\rho},\uv)
=e^{-i\vv{q}_0\cdot\vv{\rho}}\delta(\uv-\uv_0),\quad\uv_0\in\Sm^2_+,
\label{expiqrho}
\ee
where $\uv_0$ has the azimuthal angle $\va_0$ and the cosine of the polar 
angle $\mu_0$.  By collision expansion we can write $I$ as
\[
I(\vv{r},\uv)=I_b(\vv{r},\uv)+I_s(\vv{r},\uv),
\]
where $I_b$ is the ballistic term and $I_s$ is the scattered part.  
They satisfy
\[
\cases{
\uv\cdot\nabla I_b(\vv{r},\uv)+I_b(\vv{r},\uv)=0,&$z>0$,
\\
I_b(\vv{r},\uv)=f(\vv{\rho},\uv),&$z=0$, $\mu\in(0,1]$,
}
\]
and
\[
\fl
\cases{
\uv\cdot\nabla I_s(\vv{r},\uv)+I_s(\vv{r},\uv)
=\varpi\int_{\Sm^2}p(\uv,\uv')I_s(\vv{r},\uv')\,d\uv'
+S(\vv{r},\uv),&$z>0$,
\\
I_s(\vv{r},\uv)=0,&$z=0$, $\mu\in(0,1]$,
}
\]
where
\[
S(\vv{r},\uv)
=\varpi\int_{\Sm^2}p(\uv,\uv')I_b(\vv{r},\uv')\,d\uv'.
\]
We also assume $I_b,I_s\to0$ as $z\to\infty$.  Let us put
\[
\uv_0=\hvv{z}.
\]
We obtain
\[
I_b(\vv{r},\uv)=e^{-i\vv{q}_0\cdot\vv{\rho}}e^{-z}\delta(\uv-\hvv{z}).
\]
We have
\[
\fl
S(\vv{r},\uv)=\frac{\varpi}{4\pi}e^{-i\vv{q}_0\cdot\vv{\rho}}e^{-z}
\sum_{l=0}^L\beta_lP_l(\mu),\qquad
\tilde{S}(\vv{q},z,\uv)=\pi\varpi\delta(\vv{q}-\vv{q}_0)e^{-z}
\sum_{l=0}^L\beta_lP_l(\mu).
\]
Furthermore we assume that $\vv{q}_0$ is parallel to the $x$-axis:
\be
\vv{q}_0=q_0\hvv{x}.
\label{vectorq0}
\ee
We obtain
\ba
K^{m'}(\xi,\vv{q})
&=&
-\int_{\Sm^2}\int_0^{\infty}
e^{-\hat{k}_z(\xi q)z'/\xi}\tilde{S}(\vv{q},z',\uv')
\rrf{\uvk(-\xi,\vv{q})}\Phi_{-\xi}^{m'*}(\uv')\,dz'd\uv'
\\
&=&
\frac{-2\pi^{3/2}\varpi\xi}{\xi+\hat{k}_z(\xi q)}
\delta(\vv{q}-\vv{q}_0)
\sum_{l=0}^L\frac{\beta_l}{\sqrt{2l+1}}
\int_{\Sm^2}Y_{l0}(\uv')
\rrf{\uvk(-\xi,\vv{q})}\Phi_{-\xi}^{m'*}(\uv')\,d\uv'.
\ea
We note that
\ba
\int_{\Sm^2}Y_{l0}(\uv)
\rrf{\uvk(-\xi,\vv{q})}\Phi_{-\xi}^{m'*}(\uv)\,d\uv
&=&
\int_{\Sm^2}\left(\irrf{\vv{\uvk}(-\xi,\vv{q})}Y_{l0}(\uv)
\right)\Phi_{-\xi}^{m'*}(\uv)\,d\uv
\\
&=&
\sum_{m''=-l}^ld_{0m''}^l(\theta_{\uvk})
\int_{\Sm^2}\Phi_{-\xi}^{m'*}(\uv)Y_{lm''}(\uv)\,d\uv
\\
&=&
\sum_{m''=-l}^ld_{0m''}^l(\theta_{\uvk})
\sqrt{(2l+1)\pi}(-1)^{m'}\delta_{m'm''}g_l^{m'}(-\xi)
\\
&=&
\chi_{[0,l]}(|m'|)d_{0m'}^l(\theta_{\uvk})
\sqrt{(2l+1)\pi}(-1)^lg_l^{m'}(\xi).
\ea
Therefore,
\[
\fl
K^{m'}(\xi,\vv{q}_0)
=\check{K}^{m'}(\xi,\vv{q})\delta(\vv{q}-\vv{q}_0),\quad
\check{K}^{m'}(\xi,\vv{q}_0)
=\frac{-2\pi^2\varpi \xi}{\xi+\hat{k}_z(\xi q_0)}
\sum_{l=|m'|}^L(-1)^l\beta_ld_{0m'}^l(\theta_{\uvk})g_l^{m'}(\xi).
\]
This implies that $c_{lm}(\vv{q})$ have the form
\[
c_{lm}(\vv{q})=\check{c}_{lm}(\vv{q}_0)\delta(\vv{q}-\vv{q}_0).
\]

Since $\check{K}^{m'}(\xi,\vv{q}_0)$ is independent of $\va_{\vv{q}_0}$ and 
$\check{K}^{-m'}=\check{K}^{m'}$, we can write the key $F_N$ equation as
\bea
\fl
\sum_{m=0}^{l_{\rm max}}
\sum_{\alpha=0}^{\left\lfloor(l_{\rm max}-m)/2\right\rfloor}
\left[A_{m+2\alpha,m}^{m'}(\xi,q_0)+(1-\delta_{m0})(-1)^m
A_{m+2\alpha,-m}^{m'}(\xi,q_0)\right]
\check{c}_{m+2\alpha,m}(q_0)
\nonumber \\
=\check{K}^{m'}(\xi,\vv{q}_0).
\label{Btilde2}
\eea
The number of columns of the matrix $\{A(q_0)\}_{\xi^{m'},lm}
=A_{lm}^{m'}(\xi,q_0)$ is $N_{\rm tot}$, where
\[
N_{\rm tot}
=\sum_{m=0}^{l_{\rm max}}N_{\rm col}^m
=\cases{
\frac{(l_{\rm max}+2)^2}{4},&$l_{\rm max}$ even,
\\
\frac{(l_{\rm max}+1)(l_{\rm max}+3)}{4},&$l_{\rm max}$ odd,
}
\]
where $N_{\rm col}^m=\left\lfloor(l_{\rm max}-m)/2\right\rfloor+1$.  
We choose the number of rows so that $A(q_0)$ becomes square.  
For this purpose, 
different collocation schemes have been proposed \cite{Garcia-Siewert81,
Garcia-Siewert85,Garcia-Siewert98,McCormick-Sanchez81}.  Here we take, in 
addition to discrete eigenvalues $\xi_j=\nu_{j-1}^{m'}$ ($j=1,\dots,M^{m'}$), 
$N_{\rm col}^{m'}-M^{m'}$ points according to
\be
\xi_j=\cos\left(\frac{\pi}{2}\frac{j-M^{m'}}{N_{\rm col}^{m'}-M^{m'}+1}
\right),\quad
j=M^{m'}+1,\dots,N_{\rm col}^{m'}.
\label{conteigencos}
\ee
The number of components of the vector $\{\check{\vv{K}}(q_0)\}_{\xi^{m'}}
=\check{K}^{m'}(\xi,\vv{q}_0)$ is $N_{\rm tot}$.

The hemispheric flux $J_+(\vv{\rho};\vv{q}_0)$ exiting the boundary is
\ba
J_+(\vv{\rho};\vv{q}_0)
&=&
\int_0^{2\pi}\int_0^1\mu I(\vv{\rho},0,-\uv)\,d\mu d\va
\\
&\approx&
\frac{1}{4\pi^{3/2}}e^{-i\vv{q}_0\cdot\vv{\rho}}\sum_{l=0,2,\dots}
\sqrt{2l+1}\check{c}_{l0}(\vv{q}_0)\int_0^1\mu P_l(\mu)\,d\mu.
\ea
Here for even $l$
\[
\int_0^1\mu P_l(\mu)\,d\mu
=\frac{-(-1)^{l/2}(l-1)!!}{l!!(l-1)(l+2)}
=\frac{-(-1)^{l/2}l!}{2^l(l-1)(l+2)\left[\left(\frac{l}{2}\right)!\right]^2}.
\]
Therefore we obtain
\be
J_+(\vv{\rho};\vv{q}_0)
\approx
\frac{1}{4\pi^{3/2}}e^{-i\vv{q}_0\cdot\vv{\rho}}\sum_{l=0,2,\dots}
\frac{\sqrt{2l+1}(-1)^{1+l/2}l!}
{2^l(l-1)(l+2)\left[\left(\frac{l}{2}\right)!\right]^2}\check{c}_{l0}(q_0).
\label{hemisphericJ}
\ee
Let us express the absolute value as
\be
J_+(q_0)=\left|J_+(\vv{\rho};\vv{q}_0)\right|.
\label{hemisphericJabs}
\ee

The algorithm of the three-dimensional $F_N$ method can be summarized as 
follows.
\paragraph{Step 1.}
The integral over $\mu$ in (\ref{Aq}) is done using the Golub-Welsch 
algorithm \cite{Golub-Welsch} of the Gauss-Legendre quadrature with 
points $\mu_n$ and weights $w_n$ ($n=1,2,\dots,N_{\mu}$). The integral 
over $\va$ in (\ref{Aq}) is computed using the trapezoid rule with points 
$\va_j=2\pi j/N_{\va}$ ($j=0,1,\dots,N_{\va}$). We use eigenvalues of the 
matrix $B(m')$ in (\ref{Bmat}) for $\xi_j^{m'}$ corresponding to discrete 
eigenvalues and use (\ref{conteigencos}) for $\xi_j^{m'}$ corresponding 
to the continuous spectrum. We calculate $P_l^m(\mu_n)$ and 
$g_l^m(\xi_j^{m'})$ with recurrence relations. The polynomials $g_l^m(\xi)$ 
are evaluated 
according to \cite{Garcia-Siewert89,Garcia-Siewert90}.  That is, when $\xi$ 
is a discrete eigenvalue, we obtain $g_l^m(\xi)$ starting with a large degree 
using backward recursion.  For $\xi$ in the continuous spectrum, we begin 
with the initial term and successively obtain $g_l^m(\xi)$ using the 
three-term recurrence relation (\ref{chandra1}).
\paragraph{Step 2.}
The analytically continued Wigner $d$-matrices are computed using the 
recurrence relation. See \ref{Wignerpyramid}.
\paragraph{Step 3.}
We compute the double integrals in (\ref{Aq}). In the function $g^{m'}$, we 
compute $p_l^m(\mu)$ by using the recurrence relation (\ref{precurr}). 
The computation time for each double integral grows as $N_{\mu}N_{\va}$.
\paragraph{Step 4.}
The coefficients $\check{c}_{lm}(q_0)$ are obtained from the linear system 
(\ref{Btilde2}) with the $N_{\rm tot}\times N_{\rm tot}$ matrix $A(q_0)$ 
and the vector $\check{\vv{K}}(q_0)$ of length $N_{\rm tot}$.
\paragraph{Step 5.}
Once $\check{c}_{lm}(q_0)$ are obtained, $J_+(\vv{\rho};\vv{q}_0)$ is 
immediately calculated by using (\ref{hemisphericJ}).

\begin{rmk}
The computation time is dominated by the integral in (\ref{Aq}), which does 
not exist in the method of rotated reference frames (\ref{mrrf}). For a given 
$\vv{q}_0$, the computation time for the double integrals grows as 
$O(l_{\rm max}^5N_{\mu}N_{\phi})$ whereas the computation time of $J_+(q_0)$ 
scales as $O(l_{\rm max}^5)$ in the method of rotated reference frames.
\end{rmk}

For numerical calculation, let us set the absorption and scattering 
coefficients to
\[
\mu_a=0.05,\qquad\mu_s=100.
\]
We set the scattering asymmetry parameter to $\mathrm{g}=0.9$ and 
$\mathrm{g}=0.01$ (almost isotropic).  
Although the unit of length has been $1/\mu_t$, we take the transport mean 
free path $\ell^*=1/(\mu_t-\mu_s\mathrm{g})$ to be the unit of length in the 
figures. 

In Figs.~\ref{fig1} and \ref{fig2}, $J_+(q_0)$ in (\ref{hemisphericJabs}) is 
plotted as a function of the spatial frequency $q_0$. The $F_N$ result is 
compared with Monte Carlo simulation and 
the method of rotated reference frames. 
In Monte Carlo simulation $10^8$ particles were used.  To obtain Monte Carlo 
simulation for structured illumination, Fourier transform was 
performed to results from Monte Carlo simulation for the delta-function 
source \cite{LK12e}.  Monte Carlo simulation assumed the Henyey-Greenstein 
model for the scattering phase function.  The method of rotated 
reference frames for structured illumination \cite{LK12c,LK12e}
is summarized in \ref{mrrf}. 

The scattering asymmetry parameter 
$\mathrm{g}=0.01$ in Fig.~\ref{fig1} and $\mathrm{g}=0.9$ in Fig.~\ref{fig2}. 
We set $L=l_{\rm max}$. For both the $F_N$ method and the method of 
rotated reference frames we consider $l_{\rm max}=9$ and $25$.  
In Fig.~\ref{fig1}, the three methods agree reasonably well for 
$l_{\rm max}=9$. When we increase $l_{\rm max}$ aiming at more accuracy, 
however, $J_+$ from the method of rotated reference frames becomes unstable. 
Note that in this case scattering is almost isotropic and discrete 
eigenvalues are rather close to $1$. Hence we have 
$\nu-\mu\hat{k}_z(\nu q_0)<0$ (see (\ref{rotatedphi})) 
for relatively small $q_0$. 
In Fig.~\ref{fig2}, the result from the 3D $F_N$ method has a jump near 
$q_0=3.7$ for $l_{\rm max}=9$ because this $l_{\rm max}$ is not sufficiently 
large in this case. A smooth curve is obtained if large enough $l_{\rm max}$ 
is used as shown in the right panel of Fig.~\ref{fig2} for $l_{\rm max}=25$. 

\begin{figure}[ht]
\begin{center}
\includegraphics[width=0.45\textwidth]{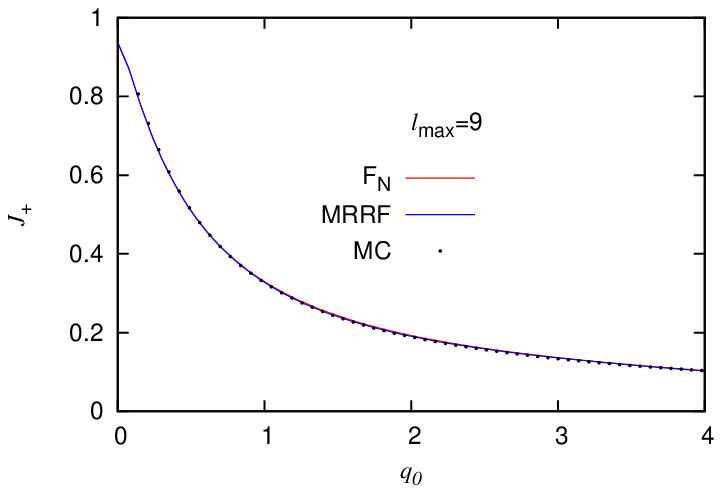}
\hspace{2mm}
\includegraphics[width=0.45\textwidth]{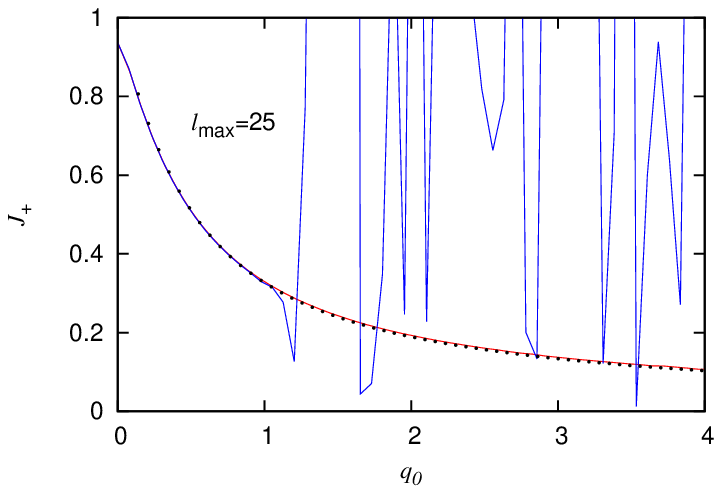}
\end{center}
\caption{
The exitance (\ref{hemisphericJabs}) is plotted against $q_0$ for 
$\mu_a=0.05$, $\mu_s=100$, and $\mathrm{g}=0.01$.  The unit of length is 
$\ell^*$.  For the $F_N$ method and the method of rotated reference frames 
(MRRF) we set (Left) $l_{\rm max}=9$ and (Right) $l_{\rm max}=25$.
}
\label{fig1}
\end{figure}

\begin{figure}[ht]
\begin{center}
\includegraphics[width=0.45\textwidth]{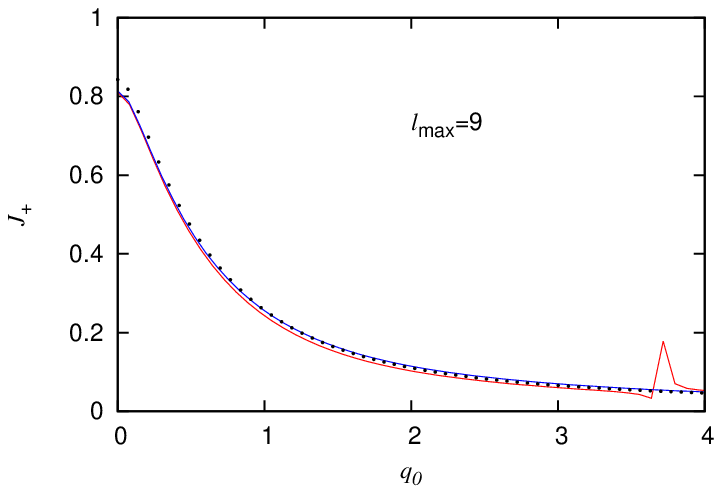}
\hspace{2mm}
\includegraphics[width=0.45\textwidth]{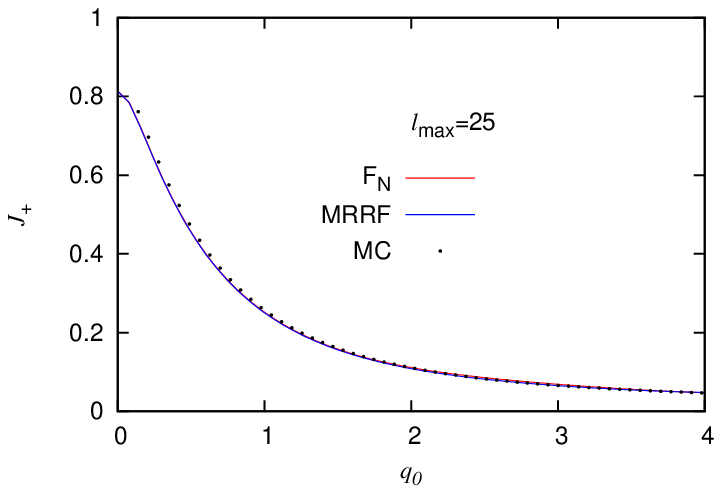}
\end{center}
\caption{
The exitance (\ref{hemisphericJabs}) is plotted against $q_0$ for 
$\mu_a=0.05$, $\mu_s=100$, and $\mathrm{g}=0.9$.  The unit of length is 
$\ell^*$.  For the $F_N$ method and the method of rotated reference frames 
(MRRF) we set (Left) $l_{\rm max}=9$ and (Right) $l_{\rm max}=25$.
}
\label{fig2}
\end{figure}

\section{Concluding remarks}
\label{concl}

The $F_N$ method is similar to the method of rotated reference frames in the 
sense that spherical-harmonic expansion is used.  However, 
in the $F_N$ method, there is no need of expanding singular eigenfunctions.  
The extension of the $F_N$ method in the half space to the slab geometry 
is straight forward. In the slab geometry, in addition to conditions such as 
(\ref{projection0}) for one plane at $z=0$, we have another set of conditions 
that corresponds to the other plane. Once the specific intensity on the 
boundary is obtained, it is also possible to compute the specific intensity 
inside the medium for the half space geometry and the slab geometry 
\cite{Siewert78}.

\ack
The author learned the $F_N$ method at the 23rd International Conference 
on Transport Theory (September 2013, Santa Fe, New Mexico).  He is grateful 
to C. E. Siewert and J. C. Schotland for fruitful discussion on collision 
expansion.  
Monte Carlo simulation was performed using the Monte Carlo solver MC 
developed by V. A. Markel.

\appendix

\section{Structured illumination with the method of rotated reference frames}
\label{mrrf}

In this section we solve (\ref{rte}) with the method of rotated reference 
frames \cite{LK12c,LK12e}.  We consider structured illumination and assume 
the source term (\ref{expiqrho}) with (\ref{vectorq0}).

We write the eigenvector of the matrix $B(M)$ in (\ref{Bmat}) corresponding 
to the eigenvalue $\nu$ as $|y_{\nu}\rangle$ ($\langle y_{\nu}|y_{\nu}\rangle
=1$).  Note that $\nu$ and $|y_{\nu}\rangle$ depend on $M$.  In the 
method of rotated reference frames, we write the specific intensity as 
a superposition of $I^{(+)}$ and $I^{(-)}$ \cite{Machida10,Panasyuk06}, where
\ba
\fl
I_M^{(+)}(\vv{r},\uv)
=e^{i\vv{q}\cdot\vv{\rho}-\hat{k}_z(\nu q)z/\nu}\sum_{l=0}^{l_{\rm max}}
\sqrt{\frac{2l+1}{h_l}}\sum_{m=-l}^l
Y_{lm}(\uv)(-1)^me^{-im\va_{\vv{q}}}\langle l|y_{\nu}\rangle
d_{mM}^l[i\tau(\nu q)],
\\
\fl
I_M^{(-)}(\vv{r},\uv)
=e^{i\vv{q}\cdot\vv{\rho}+\hat{k}_z(\nu q)z/\nu}\sum_{l=0}^{l_{\rm max}}
\sqrt{\frac{2l+1}{h_l}}\sum_{m=-l}^l
Y_{lm}(-\uv)e^{-im\va_{\vv{q}}}\langle l|y_{\nu}\rangle
d_{m,-M}^l[i\tau(\nu q)],
\ea
In the half space $\Rm^3_+$, the specific intensity is given by
\[
I(\vv{r},\uv)\approx\frac{1}{(2\pi)^2}\sum_{M=-L}^L
\sum_{\nu}\int_{\Rm^2}F_M^{(+)}I_M^{(+)}(\vv{r},\uv)\,d\vv{q},
\]
where $\sum_{\nu}$ stands for the sum over all positive eigenvalues of 
$B(M)$.  From the boundary conditions we obtain
\[
F_M^{(+)}=f_M^{(+)}(q)(2\pi)^2\delta(q_x+q_0)\delta(q_y),\qquad
f_{-M}^{(+)}(q)=(-1)^Mf_M^{(+)}(q).
\]
Here $f_M^{(+)}(q)$ are solutions to
\[
\mathcal{M}(q)f_M^{(+)}(q)=v^{(+)},\quad M\ge0,
\]
where
\[
\fl
\{\mathcal{M}(q)\}_{lm,\nu}
=\sum_{l'=0}^{l_{\rm max}}\sqrt{\frac{2l'+1}{h_{l'}}}\mathcal{B}_{ll'}^m
\langle l'|y_{\nu}\rangle\left(d_{mM}^{l'}[i\tau(\nu q)]+
(1-\delta_{M0})(-1)^Md_{m,-M}^{l'}[i\tau(\nu q)]\right),
\]
and
\[
\{v^{(+)}\}_{lm}=\delta_{m0}\sum_{l'=0}^{l_{\rm max}}
\mathcal{B}_{ll'}^0\sqrt{\frac{2l'+1}{4\pi}}.
\]
Here,
\[
\mathcal{B}_{ll'}^m
=\frac{1}{2}\sqrt{\frac{(2l+1)(2l'+1)(l-m)!(l'-m)!}{(l+m)!(l'+m)!}}
\int_0^1P_l^m(\mu)P_{l'}^m(\mu)\,d\mu.
\]
That is,
\[
\fl
I(\vv{r},\uv)
\approx\frac{1}{2\pi}\sum_{l=0}^{l_{\rm max}}\sum_{m=0}^li^m
\sqrt{\frac{2l+1}{h_l}}\left[Y_{lm}(\uv)+(1-\delta_{m0})Y_{lm}^*(\uv)\right]
K_{lm}(\vv{\rho},z),
\]
where
\ba
K_{lm}(\vv{\rho},z)
=2\pi(-i)^me^{i\vv{q}\cdot\vv{\rho}}\sum_{M\ge0}\sum_{\nu}
\langle l|y_{\nu}\rangle e^{-\hat{k}_z(\nu q_0)z/\nu}f_M^{(+)}(q_0)
\\
\times\left[d_{mM}^l[i\tau(q_0\nu)]
+(1-\delta_{M0})(-1)^Md_{m,-M}^l[i\tau(q_0\nu)]\right].
\ea
The hemispheric flux is obtained as
\bea
\fl
J_+(\vv{\rho})
=\int_0^{2\pi}\int_{-1}^0(\uv\cdot\hvv{z})I(\vv{\rho},0,\uv)\,d\mu d\va
\nonumber \\
=\int_0^{2\pi}\int_{-1}^1(\uv\cdot\hvv{z})I(\vv{\rho},0,\uv)\,d\mu d\va
-\int_0^{2\pi}\int_0^1(\uv\cdot\hvv{z})I(\vv{\rho},0,\uv)\,d\mu d\va
\nonumber \\
=\frac{1}{\sqrt{\pi h_1}}K_{10}(\vv{\rho},0)
-e^{-iq_0x}\mu_0\chi_{[0,1]}(\mu_0),
\label{j-mrrf}
\eea
where we used $I(\vv{\rho},0,\uv)=e^{-iq_0x}\delta(\uv-\uv_0)$ for 
$\mu>0$.  The expression (\ref{j-mrrf}) is used for Figs.~\ref{fig1} 
and \ref{fig2}.

\section{Analytically continued Wigner $d$-matrices}
\label{Wignerpyramid}

To compute the analytically continued Wigner $d$-matrices we use 
a pyramid scheme with recurrence relations \cite{Blanco97}. We begin with 
$d^0_{00} [i\tau(x)] \,(=1)$, $d^1_{00}[i\tau(x)]$, $d^1_{1-1}[i\tau(x)]$, 
$d^1_{10}[i\tau(x)]$, and $d^1_{11}[i\tau(x)]$:
\[
\fl
d^1_{00}=\sqrt{1+x^2},\quad
d^1_{1-1}=\frac{1-\sqrt{1+x^2}}{2},\quad
d^1_{10}=-i\frac{x}{\sqrt{2}},\quad
d^1_{11}=\frac{1+\sqrt{1+x^2}}{2}.
\]
Let us we increase $l$ iteratively up to $l_{\rm max}$.  For each value of 
$l$, we first compute $d^l_{mm'}[i\tau(x)]$ ($m=0,\dots,l-2;m'=-m,\dots,m$) 
according to
\ba
\fl
d^l_{mm'}=
\frac{l(2l-1)}{\sqrt{(l^2-m^2)(l^2-{m'}^2)}}
\\
\times
\left[\left(d^1_{00}-\frac{mm'}{l(l-1)}\right)d^{l-1}_{mm'}
-\frac{\sqrt{\left[(l-1)^2-m^2\right]\left[(l-1)^2-{m'}^2\right]}}{(l-1)(2l-1)}
d^{l-2}_{mm'}\right].
\ea
We obtain $d^l_{ll}[i\tau(x)]$ and $d^l_{l-1,l-1}[i\tau(x)]$ as
\[
d^l_{ll}=d^1_{11}d^{l-1}_{l-1,l-1},
\qquad
d^l_{l-1,l-1}=(ld^1_{00}-l+1)d^{l-1}_{l-1,l-1},
\]
and $d^l_{lm'}[\tau(x)]$ $(m'=l-1,\dots,-l)$ as
\[
d^l_{lm'}=-i\sqrt{\frac{l+m'+1}{l-m'}}
\sqrt{\left|\frac{d^1_{1-1}}{d^1_{11}}\right|}d^{l}_{l,m'+1}.
\]
With the relation
\[
d^l_{l-1,m'}=-i\frac{ld^1_{00}-m'}{ld^1_{00}-m'-1}
\sqrt{\frac{l+m'+1}{l-m'}}\sqrt{\left|\frac{d^1_{1-1}}{d^1_{11}}\right|}
d^{l}_{l-1,m'+1},
\]
we have $d^l_{l-1,m'}[i\tau(x)]$ ($m'=l-2,\dots,1-l$).  Other functions 
$d^l_{mm'}[i\tau(x)]$ are obtained by using the symmetry properties
\[
d^{l}_{mm'}=d^{l}_{-m',-m}=(-1)^{m+m'}d^{l}_{-m,-m'}=(-1)^{m+m'}d^{l}_{m'm}.
\]

\end{document}